\documentclass[11pt,a4paper]{amsart}
\usepackage{amsmath, amscd, amsfonts, amssymb, amsthm, latexsym, url, color, todonotes, bm, framed, rotating} 
\usepackage{mathrsfs}

\usepackage{tikz-cd}
\usepackage{cite}
\usepackage{hyperref}

\usepackage{graphicx}
\usepackage[left=1.23in, right=1.23in, top=1in, bottom=1.3in, includefoot, headheight=13.6pt]{geometry}
\input{xy}
\xyoption{all}

\newtheorem{theorem}{Theorem}[section]
\newtheorem{lemma}[theorem]{Lemma}
\newtheorem{corollary}[theorem]{Corollary}
\newtheorem{proposition}[theorem]{Proposition}
\newtheorem{definition}[theorem]{Definition}

\theoremstyle{definition}
\newtheorem{remark}[theorem]{Remark}

\newcommand{\xysquare}[8]{
\[\xymatrix{
#1 \ar@{#5}[r] \ar@{#6}[d] & #2 \ar@{#7}[d]\\
#3 \ar@{#8}[r] & #4
}\]
}

\newcommand{\bb}{\mathbb}

\newcommand{\comment}[1]{}

\newcommand{\into}{\hookrightarrow}

\newcommand{\onto}{\twoheadrightarrow}

\renewcommand{\phi}{\varphi}

\newcommand{\cal}{\mathcal}
\renewcommand{\hat}{\widehat}
\renewcommand{\frak}{\mathfrak}

\renewcommand{\tilde}{\widetilde}

\renewcommand{\projlim}{\varprojlim}

\renewcommand{\inf}{{\mathrm{inf}}}

\begin{document}
\font\myfont=cmr12 at 12pt
\title{Integral $p$-adic Hodge theory of formal schemes in low ramification}

\begin{abstract}
We prove that for any proper smooth formal scheme $\frak X$ over $\cal O_K$, where $\cal O_K$ is the ring of integers in a complete discretely valued non-archimedean extension $K$ of $\bb Q_p$ with perfect residue field $k$ and ramification degree $e$, the $i$-th Breuil--Kisin cohomology group and its Hodge--Tate specialization admit nice decompositions when $ie<p-1$. Thanks to the comparison theorems in the recent works of Bhatt, Morrow and Scholze \cite{BMS1}, \cite{BMS2}, we can then get an integral comparison theorem for formal schemes when the cohomological degree $i$ satisfies $ie<p-1$, which generalises the case of schemes under the condition $(i+1)e<p-1$ proven by Fontaine and Messing in \cite{Fontaine--Messing} and Caruso in \cite{caruso2008conjecture}.\\

\end{abstract}
\author{Yu Min}
\maketitle
\tableofcontents

\section{Introduction}
In this paper, we study the $A_{\inf}$-cohomology theory and the Breuil--Kisin cohomology theory constructed respectively in \cite{BMS1}, \cite{BMS2}, now unified as prismatic cohomology in \cite{bhatt2019prisms}. 

Let $\cal O_K$ always be the ring of integers in a complete discretely valued non-archimedean extension $K$ of $\bb Q_p$ with perfect residue field $k$ and ramification degree $e$. Our first main result is the following:

\begin{theorem}[Theorem \ref{HT}, Theorem \ref{comp}]\label{main}
Let $\frak X$ be a proper smooth formal scheme over $\cal O_K$. Let $\cal O_C$ be the ring of integers in a complete algebraically closed non-archimedean extension $C$ of $K$ and $X$ be the adic generic fibre of $\bar{\frak X}:=\frak X\times_{{\rm Spf}(\cal O_K)}{\rm Spf}(\cal O_C)$. Assume $ie<p-1$. Then there is an isomorphism of $\frak S=W(k)[[u]]$-modules
\[
H^i_{\frak S}(\frak X_{})\cong H^i_{\rm \acute et}(X, \bb Z_p)\otimes_{\bb Z_p}\frak S
\]
where $H^i_{\frak S}(\frak X_{}):=H^i(R\Gamma_{\frak S}(\frak X))$ is the Breuil--Kisin cohomology of $\frak X$. Consequently, we also have 
\[
H^i_{A_{\inf}}(\bar{\frak X})\cong H^i_{\rm \acute et}(X, \bb Z_p)\otimes_{\bb Z_p}A_{\inf},
\]
where $H^i_{A_{\inf}}(\bar{\frak X}):=H^i(R\Gamma_{A_{\inf}}(\bar{\frak X}))$ is the $A_{\inf}$-cohomology of $\bar{\frak X}$. Similarly under the same assumption $ie<p-1$, there is an isomorphism of $\cal O_K$-modules
\[
H^i_{\rm HT}(\frak X_{})\cong H^i_{\rm \acute et}(X, \bb Z_p)\otimes_{\bb Z_p}\cal O_K,
\]
and an isomorphism of $\cal O_C$-modules
\[
H^i_{\rm HT}(\bar{\frak X})\cong H^i_{\rm \acute et}(X, \bb Z_p)\otimes_{\bb Z_p}\cal O_C,
\]
where $H^i_{\rm HT}(\frak X):=H^i(R\Gamma_{\rm HT}(\frak X))$(resp. $H^i_{\rm HT}(\bar{\frak X}):=H^i(R\Gamma_{\rm HT}(\bar{\frak X}))$) is the Hodge--Tate cohomology\footnote{
The Hodge--Tate cohomology of $\bar{\frak X}$ satisfies: $R\Gamma_{\rm HT}(\bar{\frak X})\simeq R\Gamma_{A_{\inf}}(\bar{\frak X})\otimes_{A_{\inf},\theta}^{\bb L}\cal O_C$. We also call $R\Gamma_{\rm HT}(\frak X):=R\Gamma_{\frak S }(\frak X)\otimes_{\frak S}^{\bb L}\cal O_K$ the Hodge--Tate cohomology of $\frak X$, which may not be a standard notion.} of $\frak X$ (resp. $\bar {\frak X}$).

\end{theorem}


\begin{remark}
Note that the definition of Breuil--Kisin modules (see Definition \ref{BK}) in \cite{BMS1}, \cite{BMS2} is slightly more general than the original definition given by Kisin in \cite{kisin2006crystalline}. The difference lies in the existence of $u$-torsion (note that $\frak S=W(k)[[u]]$ is a two dimensional regular local ring). However, the theorem above shows that the Breuil--Kisin cohomology theory constructed by Bhatt, Morrow and Schloze does take values in the category of Breuil--Kisin modules in a traditional sense, at least when $ie<p-1$. 
\end{remark}

Unfortunately, we can not give any canonical isomorphisms between these modules. Our method only enables us to compare the module structure. The proof of this theorem relies essentially on the existence of the Breuil--Kisin cohomology and the construction of the $A_{\inf}$-cohomology in \cite{BMS1} by using the $L{\eta}$-functor and the pro-$\rm \acute e$tale site, which presents a close relation between $A_{\inf}$-cohomology and $p$-adic $\rm \acute e$tale cohomology. In fact, the $L{\eta}$-functor provides us with two morphisms between $H^i_{A_{\inf}}(\bar{\frak X})$ (resp. $H^i_{\rm HT}(\bar{\frak X})$) and $H^i_{\rm \acute et}(X, \bb Z_p)\otimes_{\bb Z_p}A_{\inf}$ (resp. $H^i_{\rm \acute et}(X, \bb Z_p)\otimes_{\bb Z_p}\cal O_C$), whose composition in both direction is $\mu^i$ (resp. $(\zeta_p-1)^i$). For the definitions of $\mu$ and $\zeta_p$, see Definition \ref{mu}. 

Note that $H^i_{\rm HT}(\bar{\frak X})$ is just the base change of $H^i_{\rm HT}({\frak X})$ along the natural injection $\cal O_K\to \cal O_C$. We can then directly verify the statement about the Hodge--Tate cohomology groups in Theorem \ref{main} by studying the two morphisms provided by the $L\eta$-functor.

For the part concerning the Breuil--Kisin cohomology groups, we need to prove some torsion-free results. Namely, when $ie<p-1$, the Breuil--Kisin cohomology group $H^{i+1}_{\frak S}(\frak X)$ is $E(u)$-torsion-free (equivalently, $u$-torsion-free), where $E(u)\in \frak S$ is the Eisenstein polynomial for a fixed uniformizer $\pi$ in $\cal O_K$. Moreover for any positive integer $n$, we have $H^i_{\frak S}(\frak X_{})/p^n$ is also $E(u)$-torsion-free.
~\\




As a consequence of Theorem \ref{main}, we can get an integral comparison theorem about $p$-adic $\rm \acute e$tale cohomology and crystalline cohomology both in the unramified case and ramified case, which generalises the case of schemes studied by Fontaine and Messing in \cite{Fontaine--Messing} and Caruso in \cite{caruso2008conjecture}. This is actually the main motivation of this work. Before we state our result, we give some background about integral comparison theorems.
~\\

\noindent \textbf{Integral $p$-adic Hodge theory.} For a proper smooth (formal) scheme $\frak X$ over $\cal O_K$, we can consider the de Rham cohomology $H^i_{\rm dR}(\frak X/\cal O_K)$ of $\frak X$, the $p$-adic \'etale cohomology $H^i_{\rm {\acute et}}(X, \bb Z_p)$ of the geometric (adic) generic fiber $X$ and the crystalline cohomology $H^i_{\rm crys}(\frak X_k/W(k))$ of the special fiber $\frak X_k$. Integral $p$-adic Hodge theory then studies the relations of these cohomology theories. 

The first result concerning integral comparison was given by Fontaine and Messing in \cite{Fontaine--Messing}. 
\begin{theorem}[\hspace{1sp}\cite{Fontaine--Messing}]
Let $X$ be a proper smooth scheme over $W(k)$ and $X_n=X\times_{{\rm Spec}(W(k))}{\rm Spec}(W_n(k))$, where $k$ is a perfect field of characteristic $p$. Let $G_{K_0}$ denote the absolute Galois group of $K_0=W(k)[\frac {1}{p}]$. Then for any integer $i$ such that $0\leq i\leq p-2$,  there exists a natural isomorphism of $G_{K_0}$-modules
\[
T_{\rm crys}(H^i_{\rm dR}(X_n))\simeq H^i_{\rm \acute et}(X_{\bar K_0}, \bb Z_p/p^n)
\]
where $T_{\rm crys}$ is a functor from the category of torsion Fontane--Laffaille modules to the category of $\bb Z_p[G_{K_0}]$-modules, which preserves invariant factors.
\end{theorem}
Note that $H^i_{\rm dR}(X_n)\cong H^i_{\rm crys}(X_k/W_n(k))$. Here we have used implicitly that $H^i_{\rm dR}(X_n)$ is in the category of torsion Fontane--Laffaille modules, which is actually one of the main difficulties. The proof of Fontaine--Messing's theorem relies on syntomic cohomology which acts as a bridge connecting $p$-adic $\rm \acute e$tale cohomology and crystalline cohomology.  

Recall that rational $p$-adic Hodge theory provides an equivalence between the category of crystalline representations and the category of (weakly) admissible filtered $\phi$-modules.  The idea of Fontane--Laffaille's theory is to try to classify $G_{K_0}$-stable $\bb Z_p$-lattices in a crystalline representation $V$ by $\phi$-stable $W(k)$-lattices in $D$ satisfying some conditions, where $D$ is the corresponding admissible filtered $\phi$-module. 

To generalize Fontane--Laffaille's theory to the semi-stable case, Breuil introduced the ring $S$ and related categories of $S$-modules in order to add a monodromy operator. He has also obtained an integral comparison result in the unramified case when $i<p-1$ in \cite{breuil1998cohomologie}. Later, this result was generalized to the case that $e(i+1)<p-1$ by Caruso in \cite{caruso2008conjecture}.

\begin{theorem}[\hspace{1sp}\cite{breuil1998cohomologie} \cite{caruso2008conjecture}]
Let $X$ be a proper and semi-stable scheme over $\cal O_K$. Let $X_n$ be $X\times_{{\rm Spec}(\cal O_K)}{\rm Spec}(\cal O_K/p^n)$. Fix a non-negative integer $r$ such that $er<p-1$. Then there exists a canonical isomorphism of Galois modules
\[
H^i_{{\rm \acute e}t}(X_{\bar K}, \bb Z/p^n\bb Z)(r)\cong T_{\rm st*}(H^i_{\rm log-crys}(X_n/(S/p^nS)))
\]
for any $i<r$.
\end{theorem}
$T_{\rm st*}$ is a functor from the category ${\rm Mod}^{r,\phi}_{/S_{\infty}}$ (see Definition \ref{classic definition}) to the category of $\bb Z_p[G_K]$-modules, which preserves invariant factors. The proof also relies on the use of syntomic cohomology. One of the main difficulties in their proof is to show that $H^i_{\rm log-crys}(X_n/(S/p^nS))$ is in the category ${\rm Mod}^{r,\phi}_{/S_{\infty}}$, in particular, to show that $H^i_{\rm log-crys}(X_1/(S/pS))$ is finite free over $S/pS$. 

\begin{remark}
One crucial point of Breuil's theory is that it highly depends on the restriction $r\leq p-1$ which is rooted in the fact that the inclusion $\phi({\rm Fil^r}S)\subset p^rS$ is true only when $r\leq p-1$. One way to remove this restriction is to consider Breuil--Kisin modules. In fact, one of the main motivations of $A_{\inf}$-cohomology theory is to give a cohomological construction of Breuil--Kisin modules. The techniques in \cite{BMS1} can not directly give the desired Breuil--Kisin cohomology. However, this goal is achieved in \cite{BMS2} by using topological cyclic homology and in \cite{bhatt2019prisms} by defining prismatic site in a more general setting.
\end{remark}

Recently, Bhatt, Morrow and Scholze have obtained a more general result about the relation between $p$-adic \'etale cohomology and crystalline cohomology in \cite{BMS1} by using $A_{\inf}$-cohomology. Their result does not impose any restriction on the ramification degree, roughly saying that the torsion in the crystalline cohomology gives an upper bound for the torsion in the $p$-adic \'etale cohomology.

As we have said, by studying $A_{\inf}$-cohomology and its descent Breuil--Kisin cohomology, we can generalize the results of Fontaine--Messing, Breuil and Caruso to the case of formal schemes, at least in the good reduction case.

\begin{theorem}[Theorem \ref{unrcom}, Theorem \ref{final}]\label{formal}
Let $\frak X_{}$ be a proper smooth formal scheme over $\cal O_K$. Let $C$ be a complete algebraically closed non-archimedean extension of $K$ and $\bar {\frak X}:=\frak X\times_{{\rm Spf}(\cal O_K)}{\rm Spf}(\cal O_C)$. Write $X$ for the adic generic fiber of $\bar{\frak X}$. Then when $ie<p-1$, there is an isomorphism of $W(k)$-modules ${H^i_{{\rm\acute {e}t}}}(X,\bb Z_p)\otimes_{\mathbb Z_p}W(k)\cong H^i_{\rm crys}(\frak X_k/W(k))$.
\end{theorem}

We will study the unramified case and the ramified case in different ways. For the proof in the unramified case, we need the following theorem:

\begin{theorem}[Theorem \ref{unr}]\label{UNR}
With the same assumptions as the theorem above, when $e=1$, we have 
\[
{\rm length}_{\bb Z_p}(H^i_{\rm \acute et}(X, \bb Z_p)_{\rm tor}/p^m))\geq {\rm length}_{W(k)}(H^i_{\rm crys}(\frak X_k/W(k))_{\rm tor}/p^m)
\]
for any $i<p-1$ and any positive integer $m$.
\end{theorem}
In fact, we first compare Hodge-Tate cohomology to Hodge cohomology by proving that the truncated Hodge-Tate complex of sheaves $\tau^{\leq p-1}\tilde\Omega_{\bar{\frak X}}$ is formal in this case, i.e. there is an isomorphism $\tau^{\leq p-1}\tilde\Omega_{\bar{\frak X}}\simeq \bigoplus_{i=0}^{p-1}\cal H^i(\tilde\Omega_{\bar{\frak X}})[-i]$. We then study the Hodge-to-de Rham spectral sequence to relate Hodge cohomology to de Rham cohomology. By Theorem \ref{main}, we can finally relate de Rham (or crystalline) cohomology to $p$-adic \'etale cohomology. Note that the theorem above gives a converse to Theorem \ref{torsion} in \cite{BMS1}, which implies that $H^n_{{\rm\acute {e}t}}(X,\bb Z_p)$ and $H^n_{\rm crys}(\frak X_k/W(k))$ have the same invariant factors.

In the ramified case, the integral comparison theorem follows directly from Theorem \ref{main} and Theorem \ref{Main}.

\begin{remark}
The $A_{\inf}$-cohomology theory in the semi-stable case has been studied in \cite{cesnavicius2017a_}. The Breuil--Kisin cohomology might be also generalised to the semi-stable case by using the prismatic site. Then one could also hope to generalize Theorem \ref{main} and Theorem \ref{formal} to the semi-stable case.
\end{remark}

We also remark that although the result in the ramified case can recover that in the unramified case, the method used in the unramified case can lead to the following theorem concerning the Hodge-to-de Rham spectral sequence and integral comparison result for all cohomological degrees.

\begin{theorem}[Theorem \ref{dimension}, Corollary \ref{de}]\label{alln}
Let $\frak X_{}$ be a proper smooth formal scheme over $W(k)$, where $k$ is a perfect field of characteristic $p$. Let $C$ be a complete algebraically closed non-archimedean extension of $W(k)[1/p]$ and $\cal O_C$ be its ring of integers. Let $\bar{\frak X}=\frak X_{}\times_{{\rm Spf}(W(k))}{\rm Spf}(\cal O_C)$ and write $X$ for the adic generic fiber of $\bar{\frak X}$. Assume the relative dimension of $\frak X$ satisfies ${\rm dim}\frak X<p-1$. Then we have the following results:
\begin{enumerate}
\item There is an isomorphism of $W(k)$-modules for all $i$
\[
{H^i_{{\rm\acute {e}t}}}(X,\bb Z_p)\otimes_{\mathbb Z_p}W(k)\cong H^i_{\rm crys}(\frak X_k/W(k)).
\]
\item The (integral) Hodge-to-de Rham spectral sequence degenerates at $E_1$-page.
\end{enumerate}
\end{theorem}
When $\frak X$ is a scheme, Theorem \ref{alln} can be deduced from \cite{Fontaine--Messing} together with Poincar\'e duality. When $\frak X$ is a formal scheme, the comparison isomorphism in Theorem \ref{alln} can not be deduced from Theorem \ref{formal} since there is still no Poincar\'e duality for \'etale cohomology of rigid analytic varieties over $C$ with coefficient in $\bb Z/p^n$. We also want to remark that Fontaine and Messing have proved the integral Hodge-to-de Rham spectral sequence degenerates at $E_1$-page when the special fiber of the proper smooth formal scheme $\frak X$ has dimension strictly less than $p$ (cf. Remark \ref{degeneration}).

\subsection*{Acknowledgements} 
I would like to express my great gratitude to my advisor Matthew Morrow for suggesting this problem to me and many helpful discussions during the preparation of this work. I am also grateful to Christophe Breuil, Shizhang Li and Takeshi Tsuji for their valuable comments. I also want to thank the referee for careful reading and great suggestions.
\\
\subsection*{Notations} 
Throughout this paper, let $\frak X$ be a proper smooth formal scheme over $\cal O_K$, which is the ring of integers in a complete discretely valued non-archimedean extension $K$ of $\bb Q_p$ with perfect residue field $k$ and ramification degree $e$. 

Fix $C$ a complete algebraically closed non-archimedean extension of $K$ with $\cal O_C$ its ring of integers. Let $\bar{\frak X}=\frak X_{}\times_{{\rm Spf}(\cal O_K)}{\rm Spf}(\cal O_C)$ and write $X$ for the adic generic fiber of $\bar{\frak X}$. 

Let $\frak X_k$ denote the special fiber of $\frak X$ and $\frak X_{\bar k}$ denote its base change to $\bar k$ which is the residue field of $\cal O_C$ (note that $\bar k$ is not necessarily the algebraic closure of $k$).

Define $\frak S:=W(k)[[u]]$. Fix a uniformizer $\pi$ in $\cal O_K$. We denote by $\beta$ the $W(k)$-linear map $\frak S\to \cal O_K$ sending $u$ to $\pi$, whose kernel is generated by a fixed Eisenstein polynomial $E=E(u)$ for $\pi$.

\section{Recollections on $A_{\inf}$-cohomology}\label{A}
In this section, we will simply recall the necessary ingredients for defining the $A_{\inf}$-cohomology theory in \cite{BMS1}. In fact, we will stick to the method using the pro-$\rm\acute e$tale site and the d$\rm \acute e$calage functor $L{\eta}$, which will provide us with some useful morphisms between $A_{\inf}$-cohomology groups and $p$-adic $\rm \acute e$tale cohomology groups. 
\ \

\subsection{Pro-\'etale sheaves}
We first define some sheaves on the pro-$\rm \acute e$tale site $X_{\rm pro\acute et}$. Recall that there is a natural projection map of sites
\[
\omega: X_{\rm pro\acute et}\to X_{\rm \acute et}
\]
which is defined by pulling back $U\in X_{\rm \acute et}$ to the constant tower $(\cdots\to U\to U\to X)$ in $X_{\rm pro\acute et}$. This just reflects the fact that an \'etale morphism is pro-\'etale.

\begin{definition}[\cite{scholze2013p} Section 6]
Consider the following sheaves on $X_{\rm pro\acute et}$.
\begin{enumerate}
\item The integral structure sheaf $\cal O^+_X:=\omega^*\cal O_{X_{\rm \acute et}}^+$.
\item The structure sheaf $\cal O_X:=\omega^*\cal O_{X_{\rm \acute et}}$.
\item The completed integral structure sheaf $\hat {\cal O}_X^+:=\varprojlim_n\cal O_X^+/p^n$.
\item The completed structure sheaf $\hat {\cal O}_X:=\hat {\cal O}_X^+[\frac{1}{p}]$.
\item The tilted completed integral structure sheaf $\hat {\cal O}_{X^{\flat}}^+:=\varprojlim_{\phi}\cal O_X^+/p$.
\item Fontaine's period sheaf $\bb A_{\inf,X}$, which is the derived $p$-adic completion of $W(\hat {\cal O}_{X^{\flat}}^+)$.
\end{enumerate}
\end{definition}
\begin{remark}
In \cite[Remark 5.5]{BMS1}, it has been pointed out that it is not clear whether $W(\hat {\cal O}_{X^{\flat}}^+)$ is derived $p$-adic complete. So in order to make the $A_{\inf}$-cohomology theory work well, we need to define $\bb A_{\inf,X}$ as the derived $p$-adic completion of $W(\hat {\cal O}_{X^{\flat}}^+)$, which is actually a complex of sheaves.

\end{remark}

\subsection{The $L\eta$-functor}
The other important ingredient for defining the $A_{\inf}$-cohomology is the d$\rm\acute e$calage functor, which functions as a tool to get rid of ``junk torsion''. The ``junk torsion'' exists already in Faltings' approach to $p$-adic Hodge theory in \cite{faltings1988p}. The introduction of the d\'ecalage functor is actually the main novelty of \cite{BMS1} to deal with this ``junk torsion''.

\begin{definition}[The $L\eta$-functor, \cite{BMS1} Section 6]
Let $(T, \cal O_T)$ be a ringed topos and $\cal I\subset \cal O_T$ be an invertible ideal sheaf. For any $\cal I$-torsion-free complex $C^{\bullet}\in K(\cal O_T)$, we can define a new complex $\eta_{\cal I}C^{\bullet}=(\eta_{\cal I}C)^{\bullet}\in {K}(\cal O_T)$ with terms
\[
(\eta_{\cal I}C)^i:=\{x\in C^i | dx\in \cal I\cdot C^{i+1}\}\otimes_{\cal O_T}\cal I^{\otimes i}
\]
For every complex $D^{\bullet}\in K(\cal O_T)$, there exists a strongly $K$-flat complex $C^{\bullet}\in K(\cal O_T)$ and a quasi-isomorphism $C^{\bullet}\to D^{\bullet}$. By saying strongly $K$-flat, we mean that each $C^i$ is a flat $\cal O_T$-module and for every acyclic  complex $P^\bullet\in K(\cal O_T)$, the total complex ${\rm Tot}(C^\bullet\otimes P^\bullet)$ is acyclic. In particular, $C^\bullet$ is $\cal I$-torsion free.
Then we can define 
\[
L\eta_{\cal I}: D(\cal O_T)\to D(\cal O_T)
\]
\[
L\eta_{\cal I}(D^{\bullet}):=\eta_{\cal I}(C^\bullet)
\]
\end{definition}

A concrete example is to consider a ring $A$ and a non-zero-divisor $a\in A$. If $C$ is a cochain complex of $a$-torsion free $A$-modules, we can define the subcomplex $\eta_aC$ of $C[\frac{1}{a}]$ as
\[
(\eta_{a}C)^i:=\{x\in a^iC^i : dx\in a^{i+1}\cdot C^{i+1}\}
\]
and this induces the corresponding functor $L\eta_a: D(A)\to D(A)$.




\begin{remark}
\begin{enumerate}
\item The $L\eta$-functor is not an exact functor between derived categories. For example, consider the distinguished triangle $\bb Z/p\to \bb Z/p^2\to\bb Z/p$ where the first map is induced by multiplication by $p$ on $\bb Z$ and the second map is modulo $p$. It is easy to see that $L\eta_p(\bb Z/p)=0$ and $L\eta_p(\bb Z/p^2)\neq 0$. 
\item By \cite[Proposition 6.7]{BMS1}, the $L\eta$-functor is lax symmetric monoidal.
\end{enumerate}
\end{remark}

\subsection{The $A_{\inf}$-cohomology}
We recall some basic definitions in $p$-adic Hodge theory. 
\begin{definition}[\cite{fontaine1994corps}]\label{mu}
\begin{enumerate}
\item Define $\cal O_C^{\flat}: =\projlim_{x\to x^p}\cal O_C/p$ which is called the tilt of $\cal O_C$ and $A_{\inf}:=W(\cal O_C^{\flat})$, the Witt vector ring of $\cal O_C^{\flat}$. Note that $\cal O_C^{\flat}$ is a perfect ring of characteristic $p$ and $A_{\inf}$ is equipped with a natural Frobenius map $\phi$, which is an isomorphism of rings.

\item Fix a compatible system of primitive $p$-power roots of unity $\{\zeta_{p^n}\}_{n\in \bb N}$ such that $\zeta_{p^{n+1}}^p=\zeta_{p^n}$. Under the isomorphism of multiplicative monoids $\cal O_C^{\flat}\cong \projlim_{x\to x^p}\cal O_C$, we define $\epsilon: =(1,\zeta_p, \zeta_{p^2},\cdots, \zeta_{p^n},\cdots)\in \cal O_C^{\flat}$ and $\mu: =[\epsilon]-1\in A_{\inf}$.

\item There is a map $\theta: A_{\inf}\to \cal O_C$ defined by Fontaine. The map $\theta$ is surjective and $\rm ker(\theta)$ is generated by $\xi=\mu/{\phi^{-1}(\mu)}$. After twisting with the Frobenius map, we get $\tilde \theta:=\theta\circ \phi^{-1}: A_{\inf}\to \cal O_C$, whose kernel is generated by $\tilde \xi:=\phi(\xi)=\phi(\mu)/\mu$.
\end{enumerate}
\end{definition}

Now we are ready to define the $A_{\inf}$-cohomology theory. We consider the natural projection $\nu: X_{\rm pro\acute et}\to \bar{\frak X}_{\rm zar}$, which is actually the composite $X_{\rm pro\acute et}\xrightarrow{\omega} X_{\rm \acute et}\to \bar{\frak X}_{\rm \acute et}\to \bar{\frak X}_{\rm zar}$. 

\begin{definition}[\cite{BMS1} Definition 9.1]
Define $A\Omega_{\bar{\frak X}}:=L{\eta}_{\mu}R\nu_{*}(\bb A_{\inf, X})$ and $\tilde \Omega_{\bar{\frak X}}:=L\eta_{\zeta_{p-1}}R\nu_{*}(\hat {\cal O}^+_X)$. The $A_{\inf}$-cohomology is defined to be the Zariski hypercohomology of the complex of sheaves $A\Omega_{\bar{\frak X}}$, i.e. $R\Gamma_{A_{\inf}}(\bar{\frak X}):=R\Gamma_{\rm zar}(\bar{\frak X}, A\Omega_{\bar{\frak X}})$. We can also define the Hodge--Tate cohomology $R\Gamma_{\rm HT}(\bar{\frak X}):=R\Gamma_{\rm zar}(\bar{\frak X}, \tilde \Omega_{\bar{\frak X}})$.
\end{definition}
\begin{remark}\label{commutative algebra structure}
As both $R\nu_{*}$ and the $L\eta$-functor are lax symmetric monoidal, the complex $\tilde \Omega_{\bar{\frak X}}$ is a commutative $\cal O_{\overline{\frak X}}$-algebra object in the derived category of $\cal O_{\overline{\frak X}}$-modules $D(\cal O_{\overline{\frak X}})$. For the same reason, the complex $A\Omega_{\bar{\frak X}}$ is a commutative ring in the derived category $D(\overline{\frak X}_{\rm zar}, \bb Z)$ of abelian sheaves. 
\end{remark}


The $A_{\inf}$-cohomology takes values in the category of what we call Breuil--Kisin-Fargues modules.
\begin{definition}[\cite{BMS1} Definition 4.22]
A Breuil--Kisin-Fargues module is a finitely presented $A_{\inf}$-module $M$ which becomes free over $A_{\inf}[\frac{1}{p}]$ after inverting $p$ and is equipped with an isomorphism
\[
\phi_M: M\otimes_{A_{\inf},\phi}A_{\inf}[\frac{1}{\tilde \xi}]\xrightarrow{\simeq}M[\frac{1}{\tilde\xi}].
\]
\end{definition}
The main theorem about the $A_{\inf}$-cohomology theory is the following:

\begin{theorem}[\hspace{1sp}\cite{BMS1} Theorem 14.3]\label{Main}
The complex $R\Gamma_{A_{\inf}}(\bar{\frak X})$ is a perfect complex of $A_{\inf}$-modules with a $\phi$-linear map $\phi: R\Gamma_{A_{\inf}}(\bar{\frak X})\to R\Gamma_{A_{\inf}}(\bar{\frak X})$ which becomes an isomorphism after inverting $\xi$ resp. $\tilde \xi$. The cohomology groups $H^i_{A_{\inf}}(\bar{\frak X}):=H^i(R\Gamma_{A_{\inf}}(\bar{\frak X}))$ are Breuil--Kisin-Fargues modules. Moreover, there are several comparison results:
\begin{enumerate}
\item With $\acute e$tale cohomology: $R\Gamma_{A_{\inf}}(\bar{\frak X})\otimes^{\bb L}_{A_{\inf}} A_{\inf}[1/\mu]\simeq R\Gamma_{\rm\acute et}(X, \bb Z_p)\otimes^{\bb L}_{\bb Z_p}A_{\inf}[1/\mu]$.
\item With crystalline cohomology: $R\Gamma_{A_{\inf}}(\bar{\frak X})\otimes^{\bb L}_{A_{\inf}} W(\bar k)\simeq R\Gamma_{\rm crys}(\frak X_{\bar k}/W(\bar k))$, where the map $A_{\inf}=W(\cal O_C^{\flat})\to W(\bar k)$ is induced by the natural projection $\cal O_C^{\flat}\to \bar k$ (in fact, $\cal O_C^{\flat}$ is a valuation ring with residue field $\bar k$).
\item With de Rham cohomology: $R\Gamma_{A_{\inf}}(\bar{\frak X})\otimes^{\bb L}_{A_{\inf},\theta} \cal O_C\simeq R\Gamma_{\rm dR}(\bar{\frak X}/\cal O_C)$.
\item With Hodge--Tate cohomology: $\tilde \Omega_{\bar{\frak X}}\simeq A\Omega_{\bar{\frak X}}\otimes_{A_{\inf},\tilde \theta}^{\bb L}\cal O_C$ and $R\Gamma_{A_{\inf}}(\bar{\frak X})\otimes^{\bb L}_{A_{\inf},\tilde \theta} \cal O_C\simeq R\Gamma_{\rm HT}(\bar{\frak X})$.
\end{enumerate}
\end{theorem}
\begin{corollary}\label{special}
For all $i\geq 0$, we have isomorphisms and short exact sequences
\begin{enumerate}
\item $H^i_{A_{\inf}}(\bar {\frak X})\otimes_{A_{\inf}}A_{\inf}[1/\mu]\cong H^i_{\rm \acute et}(X, \bb Z_p)\otimes_{\bb Z_p}A_{\inf}[1/\mu]$.
\item $0\to H^i_{A_{\inf}}(\bar{\frak X})\otimes_{A_{\inf}}W(\bar k)\to H^i_{\rm crys}(\frak X_{\bar k}/W(\bar k))\to Tor_1^{A_{\inf}}(H^{i+1}_{A_{\inf}}(\bar{\frak X}), W(\bar k))\to 0$.
\item $0\to H^i_{A_{\inf}}(\bar{\frak X})\otimes_{A_{\inf},\theta}\cal O_C\to H^i_{\rm dR}(\bar{\frak X}/\cal O_C)\to H^{i+1}_{A_{\inf}}(\bar{\frak X})[\xi]\to 0$.
\item $0\to H^i_{A_{\inf}}(\bar{\frak X})\otimes_{A_{\inf},\tilde\theta}\cal O_C\to H^i_{\rm HT}(\bar{\frak X})\to H^{i+1}_{A_{\inf}}(\bar{\frak X})[\tilde \xi]\to 0$.

\end{enumerate}

\end{corollary}

One of the most important applications of the $A_{\inf}$-cohomology theory is to enable us to compare \'etale cohomology to crystalline cohomology integrally without any restrictions on the degree of cohomology groups and the ramification degree of the base field. More precisely, it can be showed that the torsion in the crystalline cohomology gives an upper bound for the torsion in the $\rm \acute e$tale cohomology.

\begin{theorem}[\hspace{1sp}\cite{BMS1} Theorem 14.5]\label{torsion}
For any $n, i\geq 0$, we have the inequality
\[
{\rm length}_{W(k)}(H^i_{\rm crys}(\frak X_k/W(k))_{\rm tor}/p^n)\geq {\rm length}_{\bb Z_p}(H^i_{\rm \acute et}(X, \bb Z_p)_{\rm tor}/p^n)
\]
where $H^i_{\rm crys}(\frak X_k/W(k))_{\rm tor}$ is the torsion submodule of $H^i_{\rm crys}(\frak X_k/W(k))$ and $H^i_{\rm \acute et}(X, \bb Z_p)_{\rm tor}$ is the torsion submodule of $H^i_{\rm \acute et}(X, \bb Z_p)$.
\end{theorem}

As we have mentioned, there is a refinement of the $A_{\inf}$-cohomology, i.e. the Breuil--Kisin cohomology, which is an $\frak S$-linear cohomology and recovers the $A_{\inf}$-cohomology after base change along a faithfully flat map\footnote{To define the map $\alpha$, we fix a compatible system of $p$-power roots $\pi^{1/{p^n}}\in C$, which defines an element $\pi^{\flat}=(\pi, \pi^{1/p}, \pi^{1/p^2}, \cdots)\in \projlim_{x\mapsto x^p}\cal O_C\cong \cal O_C^{\flat}$. Then $\alpha$ is defined to send $u$ to $[\pi^{\flat}]^p$ and be the Frobenius on $W(k)$.} $\alpha: \frak S\to A_{\inf}$ (in particular, we have $(\alpha(E))=(\tilde \xi)$). The Breuil--Kisin cohomology gives a cohomological construction of Breuil--Kisin modules, which plays a very important role in integral $p$-adic Hodge theory.

The first construction of the Breuil--Kisin cohomology is given in \cite{BMS2}  by using topological cyclic homology.  Another construction is given in\cite{bhatt2019prisms} by using the prismatic site. We will not say anything about the construction of the Breuil--Kisin cohomology theory here but choose to state a similar comparison theorem as in the $A_{\inf}$ case.

Let $R\Gamma_{\frak S}(\frak X)$ denote the Breuil-Kisin cohomology attached to $\frak X$. We first recall the definition of Breuil--Kisin module which is slightly more general than the original definition due to Kisin.
\begin{definition}[\cite{BMS1} Definition 4.1]\label{BK}
A Breuil--Kisin module is a finitely generated $\frak S$-module $M$ together with an isomorphism 
\[
\phi_M: M\otimes_{\frak S,\phi}\frak S[{\frac {1}{E}}]\to M[{\frac {1}{E}}].
\]
\end{definition}
\begin{theorem}[\hspace{1sp}\cite{BMS2} Theorem 1.2]\label{theorem1.2}
The Breuil--Kisin cohomology $R\Gamma_{\frak S}(\frak X)$ of $\frak X$ is a perfect complex of $\frak S$-modules. Moreover, it is equipped with a $\phi$-linear map $\phi: R\Gamma_{\frak S}(\frak X_{})\to R\Gamma_{\frak S}(\frak X_{})$ which induces an isomorphism 
\[
R\Gamma_{\frak S}(\frak X_{})\otimes^{\bb L}_{\frak S,\phi}\frak S[\frac {1}{E}]\simeq R\Gamma_{\frak S}(\frak X_{})[\frac{1}{E}]
\]
The cohomology groups $H^i_{\frak S}(\frak X_{}):=H^i(R\Gamma_{\frak S}(\frak X_{}))$ are Breuil--Kisin modules. There are several specializations that recover other $p$-adic cohomology theories:
\begin{enumerate}
\item With $A_{\inf}$-cohomology: after base change along $\alpha: \frak S\to A_{\inf}$, it recovers $A_{\inf}$-cohomology : $R\Gamma_{\frak S}(\frak X)\otimes^{\bb L}_{\frak S,\alpha}A_{\inf}\simeq R\Gamma_{A_{\inf}}(\bar{\frak X})$.

\item With \'etale cohomology: $R\Gamma_{\frak S}(\frak X_{})\otimes^{\bb L}_{\frak S,\tilde\alpha}A_{\inf}[1/\mu]\simeq R\Gamma_{\rm \acute{e}t}(X, \bb Z_p)\otimes^{\bb L}_{\bb Z_p}A_{\inf}[1/\mu]$, where $\tilde \alpha$ is the composite $\frak S\xrightarrow{\alpha} A_{\inf}\into A_{\inf}[1/\mu]$.

\item With de Rham cohomology: $R\Gamma_{\frak S}(\frak X_{})\otimes^{\bb L}_{\frak S,\tilde\beta}\cal O_K\simeq R\Gamma_{\rm dR}(\frak X_{}/\cal O_K)$, where $\tilde \beta:=\beta\circ\phi: \frak S\to \cal O_K$.

\item With crystalline cohomology: after base change along the map $\frak S\to W(k)$ which is the Frobenius on $W(k)$ and sends $u$ to 0, it recovers the crystalline cohomology of the special fiber: $R\Gamma_{\frak S}(\frak X_{})\otimes^{\bb L}_{\frak S}W(k)\simeq R\Gamma_{\rm crys}(\frak X_k/W(k))$.
\end{enumerate}
\end{theorem}
For later convenience, we define $R\Gamma_{\rm HT}(\frak X):=R\Gamma_{\frak S}(\frak X)\otimes^{\bb L}_{\frak S,\beta}\cal O_K$ and call it the Hodge--Tate cohomology of $\frak X$. Note that there is an isomorphism: $R\Gamma_{\rm HT}(\frak X)\otimes_{\cal O_K}^{\bb L}\cal O_C\simeq R\Gamma_{\rm HT}(\overline{\frak X})$.
\begin{remark}
Note that there is a Frobenius twist appearing in the specializations above. As explained in \cite[Remark 1.4]{BMS2}, this is not artificial but contains some information about the torsion in the de Rham cohomology.
\end{remark}

\section{Lemmas in commutative algebra}\label{B}
In this section, we will recollect some results on finitely presented modules over $\cal O_C$ and prove some key lemmas that are frequently used later.

We begin with the following lemma:

\begin{lemma}[\cite{stacks-project}\href{https://stacks.math.columbia.edu/tag/0ASN}{Lemma 0ASN}]
Let $R$ be a ring. The following are equivalent:
\begin{enumerate}
\item For $a,b\in R$, either $a$ divides $b$ or $b$ divides $a$.
\item Every finitely generated ideal is principal and $R$ is local.
\item The set of ideals of $R$ are linearly ordered by inclusion.
\end{enumerate}

\end{lemma}
In particular, all valuation rings satisfy these equivalent conditions. The module structure of finitely presented modules over valuation rings is similar to that of finitely generated modules over principal ideal domains as the following lemma shows.

\begin{lemma}[\cite{stacks-project}\href{https://stacks.math.columbia.edu/tag/0ASP}{Lemma 0ASP}]\label{pid}
Let $R$ be a ring satisfying the equivalent conditions above, then every finitely presented $R$-module is isomorphic to a finite direct sum of modules of the form $R/aR$ where $a\in R$.
\end{lemma}

\begin{corollary}\label{structure}
Any finitely presented module over $\cal O_C$ is of the form $\bigoplus_{i=1}^n\cal O_C/{\pi_i}$ for some $\pi_i\in \cal O_C$. 
\end{corollary}
 
We will need to study finitely presented torsion $\cal O_C$-modules later. The main tool to deal with these modules is the length function $l_{\cal O_C}$, as used in \cite{cesnavicius2017a_}, see also \cite{bhatt2017hodge}. In particular, this length function behaves additively under short exact sequences. Usually, we use the normalized length function, i.e. $l_{\cal O_C}(\cal O_C/p)=1$.
\begin{lemma}\label{lengthequal}
Let $A$ and $B$ be base changes to $\cal O_C$ of finitely presented torsion $W(k)$-modules. If for each $m>0$, we have 
\[
l_{\cal O_C}(A/p^m)=l_{\cal O_C}(B/p^m),
\]
then $A$ is isomorphic to $B$ as $\cal O_C$-modules.
\end{lemma}
\begin{proof}
Note that $2l_{\cal O_C}(A/p)-l_{\cal O_C}(A/p^2)$ is the number of the invariant factor $p$ in $A$. This implies that the number of the invariant factor $p$ of $A$ is equal to that of $B$. By induction on $m$, it is easy to prove $A\cong B$ as $\cal O_C$-modules.

\end{proof}

Next, we want to prove the following key lemma which will be used in the comparison of Hodge--Tate cohomology and $p$-adic \'etale cohomology.
\begin{lemma}\label{lemma2}
Let $M=\bigoplus_{i=1}^m\cal O_C/{\beta^{m_i}}$ and $N=\bigoplus_{j=1}^n\cal O_C/{\beta^{n_j}}$, where $\beta\neq 0$ is in the maximal ideal $\frak m$ of $\cal O_C$ and all $m_i, n_j$ are positive integers. Suppose there are two $\cal O_C$-linear morphisms $f: M\to N$ and $g: N\to M$ such that $g\circ f=\alpha$ and $f\circ g=\alpha$, where $\alpha\in \cal O_C$ and $v(\beta)>v(\alpha)$. Then $m=n$ and the multi-sets $\{m_i\}$ and $\{n_j\}$ are the same, i.e., $M\cong N$.
\end{lemma}

In order to prove this lemma, we consider all finitely presented torsion modules over $\cal O_C$. As we have mentioned, any such module looks like $\bigoplus_{i=1}^n\cal O_C/{\pi_i}$ for some non-zero $\pi_i\in \frak m$. We call ${\rm trk}(N):=n$ the torsion-rank of $N$. Note that the torsion-rank of $N$ is equal to the dimension of $N$ base changed to the residue field of $\cal O_C$. So it is well-defined. We will also use the normalized length function $l_{\cal O_C}$ for finitely presented torsion $\cal O_C$-modules.

Now we prove a lemma concerning the torsion-rank:
\begin{lemma}\label{trk}
Let $N\into M$ be an injection of finitely presented torsion $\cal O_C$-modules. Then ${\rm trk}(N)\leq {\rm trk}(M)$. Dually if $N\onto M$ is a surjection of finitely presented torsion $\cal O_C$-modules, then ${\rm trk}(N)\geq {\rm trk}(M)$.
\end{lemma}
\begin{proof}
Write $N=\bigoplus_{i=1}^n\cal O_C/{\pi_i}$ and $M=\bigoplus_{i=1}^m\cal O_C/{\varpi_i}$.  Let $\pi$ be the smallest of the $\pi_i$ (i.e., the one with the smallest valuation), and let $\varpi$ be the largest of $\varpi_i$. Then 
\[
(\cal O_C/{\pi})^n\subset N\into M\subset (\cal O_C/\varpi)^m,
\]
which shows that $\varpi\in\pi\cal O_C$; write $\varpi=\pi x$. The composition of these maps lands in the $\pi$-torsion of $(\cal O_C/{\varpi})^m$, which is isomorphic to $(x\cal O_C/{\varpi\cal O_C})^m\cong(\cal O_C/{\pi\cal O_C})^m$. So now we have an injection $(\cal O_C/{\pi})^n\into (\cal O_C/{\pi})^m$. Taking length shows that $n\leq m$. 

If $N\onto M$ is a surjection of finitely presented torsion $\cal O_C$-modules, we can consider the injection ${\rm Hom}(M, \cal O_C/t)\into {\rm Hom}(N, \cal O_C/t)$ where $t$ is any non-zero element in $\frak m$. Then we have ${\rm trk}(M)={\rm trk}({\rm Hom}(M, \cal O_C/t))\leq {\rm trk}({\rm Hom}(N, \cal O_C/t))={\rm trk}(N)$.
\end{proof}

\begin{lemma}\label{mmlemma}
Let $g: N\to M$ be a morphism of finitely presented torsion $\cal O_C$-modules; write $N=\bigoplus_{i=1}^n\cal O_C/{\pi_i}$ and $M=\bigoplus_{i=1}^m\cal O_C/{\varpi_i}$. Assume that ${\rm ker}(g)$ is killed by some element $\alpha\in \cal O_C$ whose valuation is strictly smaller that all of the ${\pi_i}$. Then ${\rm trk}(N)\leq {\rm trk}(M)$.
\end{lemma}

\begin{proof}
By assumption ${\rm ker}(g)$ is contained in the $\alpha$-torsion $N[\alpha]$ of $N$, which is given by $N[\alpha]\cong \bigoplus_{i=1}^n\frac{\pi_i}{\alpha}\cal O_C/{\pi_i\cal O_C}$. So 
\[
N\onto N/{{\rm ker}(g)}\onto N/{N[\alpha]}\cong \bigoplus_{i=1}^n\frac{\pi_i}{\alpha}\cal O_C/{\pi_i\cal O_C}.
\]
Taking torsion-ranks, Lemma \ref{trk} for surjections shows that ${\rm trk}(N/{{\rm ker}(g)})={\rm trk}(N)$. But $N/{{\rm ker}(g)}\into M$, so Lemma \ref{trk} also shows that ${\rm trk}(N/{{\rm ker}(g)})\leq{\rm trk}(M)$.
\end{proof}
Now we are ready to prove Lemma \ref{lemma2}.

\begin{proof}[Proof of Lemma \ref{lemma2}] 
Note that the number of invariant factor $\beta^k$ in $M$ is equal to ${\rm trk}(\beta^{k-1}M)-{\rm trk}(\beta^kM)$.  By Lemma \ref{mmlemma}, we have ${\rm trk}(\beta^kM)={\rm trk}(\beta^kN)$ for any $k$. This means that the number of invariant factor $\beta^k$ in $M$ and that in $N$ are equal for any $k$. So we must have $M\cong N$.
\end{proof}

\begin{lemma}\label{torsionlemma}
Let $M=\mathcal O_C^r\oplus(\bigoplus_{i=1}^{m}{\cal O_C}/\beta^{m_i})$ and $N=\mathcal O_C^s\oplus(\bigoplus_{j=1}^{n}{ \cal O_C}/\beta^{n_j})$. Suppose there are two $\cal O_C$-linear morphisms $f : M\to N$ and $g: N\to M$ such that $g\circ f=\alpha$ and $f\circ g=\alpha$, where $\alpha\in \cal O_C$ and $v(\beta)>v(\alpha)$. Then $M\cong N$. In particular, if $M=0$, then $N=0$.
\end{lemma}
\begin{proof}
According to Lemma \ref{lemma2}, $M/{\beta^k}$ and $N/{\beta^k}$ are isomorphic for all $k$. For large enough $k$, this means the torsion submodule $M_{\rm tor}$ of $M$ is isomorphic to the torsion submodule $N_{\rm tor}$ of $N$ and also the rank of the free part of $M$ is equal to that of $N$, i.e. $r=s$. We are done.
\end{proof}

\section{Hodge--Tate cohomology}\label{C}

In this section, we study the Hodge--Tate specialization of the Breuil--Kisin cohomology and prove the isomorphism concerning Hodge--Tate cohomology groups in Theorem \ref{main} under the restriction $ie<p-1$. 

Our strategy is to first study the Hodge--Tate specialization of the $A_{\inf}$-cohomology of $\bar{\frak X}$. We can take advantage of the $L{\eta}$-construction of $A_{\inf}$-cohomology and its Hodge--Tate specilizaton. This will provide us with two morphisms which enable us to use Lemma \ref{torsionlemma}. In order to make this more precise, we need to introduce the framework of almost mathematics (derived category version) following \cite{bhatt2018specializing}. 

\begin{definition}[The pair $(\cal O_C, \frak m)$]
Let $\frak m$ denote the maximal ideal of $\cal O_C$. We say an $\cal O_C$-module $M$ is almost zero if $\frak m\cdot M=0$. A map $f: K\to L$ in $D(\cal O_C)$ is an almost isomorphism if the cohomology groups of the mapping cone of $f$ are almost zero.
\end{definition}

Now we consider the almost derived category of $\cal O_C$-modules. Precisely, there are two functors:
\[
D(\cal O_C)\xrightarrow{()^a} D(\cal O_C)^a:=D(\cal O_C)/D(k),\ \ \ \ K\mapsto K^a
\]
\[
D(\cal O_C)^a\xrightarrow{()_*} D(\cal O_C), \ \ \ \ K^a\mapsto (K^a)_*:={\rm RHom}_{\cal O_C}(\frak m, K)
\]
where the Verdier quotient $D(\cal O_C)/D(k)$ is actually the localization of $D(\cal O_C)$ with respect to almost isomorphisms. The functor $()_*$ is right adjoint to the quotient functor $()^a$.

\begin{lemma}\label{almost1}
If $C$ is spherically complete, i.e. any decreaing sequence of discs in $C$ has nonempty intersection, we have $K\simeq (K^a)_{*}$ for any perfect complex $K\in D(\cal O_C)$.
\end{lemma}
\begin{proof}
See \cite[Lemma 3.4]{bhatt2018specializing}.
\end{proof}

There are similar constructions and results in the setting of $A_{\inf}$-modules.
\begin{definition}[The pair $(A_{\inf}, W({\frak m}^{\flat}))$]
An $A_{\inf}$-module $M$ is called almost zero if $W({\frak m}^{\flat})\cdot M=0$, where $W({\frak m}^{\flat})={\rm Ker}(A_{\inf}\to W(\bar k))$. A map $f: K\to L$ in $D(A_{\inf})$ is called an almost isomorphism if the cohomology groups of the mapping cone of $f$ are almost zero.
\end{definition}
Similarly, we consider the almost derived category of $A_{\inf}$-modules. Let $D_{\rm comp}(A_{\inf})\subset D(A_{\inf})$ be the full subcategory of all derived $p$-adically complete complexes. There are two functors:
\[
D_{\rm comp}(A_{\inf})\xrightarrow{()^a}D_{\rm comp}(A_{\inf})^a:=D_{\rm comp}(A_{\inf})/D_{\rm comp}(W(k)), \ \ \ K\mapsto K^a
\]
\[
D_{\rm comp}(A_{\inf})^a\xrightarrow{()_*}D_{\rm comp}(A_{\inf}), \ \ \ K^a\to (K^a)_*:= {\rm RHom}_{A_{\inf}}(W({\frak m}^{\flat}), K)
\]
where the Verdier quotient $D_{\rm comp}(A_{\inf})^a:=D_{\rm comp}(A_{\inf})/D_{\rm comp}(W(k))$ is actually the localization of $D_{\rm comp}(A_{\inf})$ with respect to almost isomorphisms. The functor $()_*$ is also right adjoint to $()^a$.

\begin{lemma}\label{almostainf}
If $C$ is spherically complete, we have $K\simeq (K^a)_*$ for any perfect complex $K\in D_{\rm comp}(A_{\inf})$.
\end{lemma}
\begin{proof}
See \cite[Lemma 3.10]{bhatt2018specializing}.
\end{proof}

Now we are ready to study the structure of the Hodge--Tate cohomology groups. We first state a lemma about the $L\eta$-functor, which will give us two important maps connecting Hodge--Tate cohomology and $p$-adic $\rm{\acute e}$tale cohomology.
\begin{lemma}\label{bhattkey}
Let $A$ be a commutative ring and $a\in A$ be a non-zero divisor. Assume $K\in D^{[0,s]}(A)$ with $H^0(K)$ being $a$-torsion free. Then there are natural maps $L{\eta}_a(K)\to K$ and $K\to L\eta_a(K)$ whose composition in either direction is $a^s$.
\end{lemma}
\begin{proof}
This is \cite[Lemma 6.9]{BMS1}. We give the proof here.

Firstly, we choose a representative $L$ of $K$ such that $L$ is $a$-torsion free. Then we apply the truncation functor $\tau^{\leq s}$ and $\tau^{\geq 0}$ to $L$, i.e. $\tau^{\leq s}\tau^{\geq 0}L=(\cdots\to 0\to L^0/{{\rm Im}(d^{-1})^{}}\to L^1\to \cdots \to L^{s-1}\to {\rm ker}(d^n)\to 0\cdots)$. Since $K\in D^{[0,s]}(A)$, $\tau^{\leq s}\tau^{\geq 0}L$ is still isomorphic to $K$. Moreover $\tau^{\leq s}\tau^{\geq 0}L$ is still $a$-torsion free. It is easy to see that ${\rm ker}(d^n)$ is $a$-torsion free. For $L^0/{{\rm Im}(d^{-1})^{}}$, suppose $\bar x\in L^0/{{\rm Im}(d^{-1})^{}}$ is killed by $a$, then $ax\in {\rm Im}(d^{-1})$ for any lifting $x\in L^0$ of $\bar x$ and $d^0(ax)=ad^0(x)=0$. As $L^0$ is $a$-torsion free, $d^0(x)$ must be 0, which implies that $x\in {\rm ker}(d^0)$. But this also means that $H^0(L)=H^0(K)$ has $a$-torsion. So $\tau^{\leq s}\tau^{\geq 0}L$ is still $a$-torsion free and we can apply $\eta$-functor to it.

There is a natural inclusion $\eta_a(\tau^{\leq s}\tau^{\geq 0}L)\to \tau^{\leq s}\tau^{\geq 0}L$. We can define another map $\tau^{\leq s}\tau^{\geq 0}L\to \eta_a(\tau^{\leq s}\tau^{\geq 0}L)$ by multiplying by $a^s$. Then the composition of these two maps in either direction is $a^s$.
\end{proof}

We may apply Lemma \ref{bhattkey} to $A=\cal O_{\bar {\frak X}}$, $a=\zeta_p-1$ and $K=\tau^{\leq i}R\nu_*{\hat {\cal O}_{X}^+}$. In fact $\tau^{\leq i}R\nu_*{\hat {\cal O}_{X}^+}$ is in $D^{[0,i]}(\cal O_{\bar{\frak X}})$ with $H^0(\tau^{\leq i}R\nu_*{\hat {\cal O}_{X}^+})$ being $(\zeta_p-1)$-torsion-free. By the same argument in the proof of Lemma \ref{bhattkey} we can always find a representative $L$ of $\tau^{\leq i}R\nu_*{\hat {\cal O}_{X}^+}$ such that $L$ is $(\zeta_p-1)$-torsion-free and $L^s=0$ for any $s\notin [0,i]$. Then there are two natural maps which we denote by $f$ and $g$, 
 \[
 f:L\eta_{\zeta_p-1}(\tau^{\leq i}R\nu_*{\hat {\cal O}_{X}^+})\simeq \tau^{\leq i}\tilde \Omega_{\bar{\frak X}}\to \tau^{\leq i}R\nu_*{\hat {\cal O}_{X}^+}
 \]
 \[
 g:\tau^{\leq i}R\nu_*{\hat {\cal O}_{X}^+}\to \tau^{\leq i}\tilde \Omega_{\bar{\frak X}}
 \]
 whose composition in either direction is $(\zeta_p-1)^i$. The isomorphism $L\eta_{\zeta_p-1}(\tau^{\leq i}R\nu_*{\hat {\cal O}_{X}^+})\simeq \tau^{\leq i}\tilde \Omega_{\bar{\frak X}}$ is due to the commutativity of the $L\eta$ functor and the canonical truncation functor $\tau^{\leq i}$ (see \cite[Corollary 6.5]{BMS1}). Recall that for any $K\in D(\cal O_{\bar{\frak X}})$, $\tau^{\leq i}K:=(\cdots\to K^{i-1}\xrightarrow{d^{i-1}}{\rm ker}(d^{i})\to 0\to \cdots)$.
 
 Passing to sheaf cohomology, we get two natural maps
 \[
 f: \tau^{\leq i}R\Gamma_{\rm zar}(\bar{\frak X}, \tau^{\leq i} \tilde \Omega_{\bar{\frak X}})\to \tau^{\leq i}R\Gamma_{\rm zar}(\bar{\frak X}, \tau^{\leq i} R\nu_*{\hat {\cal O}}_{X}^+)
 \]
 \[
 g:\tau^{\leq i}R\Gamma_{\rm zar}(\bar{\frak X}, \tau^{\leq i} R\nu_*{\hat {\cal O}}_{ X}^+)\to \tau^{\leq i}R\Gamma_{\rm zar}(\bar{\frak X}, \tau ^{\leq i}\tilde \Omega_{\bar{\frak X}})
 \]
 whose composition in either direction is $(\zeta_p-1)^i$. Since there is an isomorphsim 
 \[
 \tau^{\leq i}R\Gamma_{\rm zar}(\bar{\frak X}, \tau^{\leq i} \tilde \Omega_{\bar{\frak X}})\simeq \tau^{\leq i}R\Gamma_{\rm zar}(\bar{\frak X},  \tilde \Omega_{\frak X})
 \]
 which is induced by the natural morphism $\tau^{\leq i} \tilde \Omega_{\bar{\frak X}}\to \tilde \Omega_{\bar{\frak X}}$, we get two maps
 \[
 f:\tau^{\leq i}R\Gamma_{\rm zar}(\bar{\frak X},  \tilde \Omega_{\bar{\frak X}})\to \tau^{\leq i}R\Gamma_{\rm zar}(\bar{\frak X}, R\nu_*{\hat {\cal O}_{X}^+})
 \]
 \[
 g:\tau^{\leq i}R\Gamma_{\rm zar}(\bar{\frak X}, R\nu_*{\hat {\cal O}_{X}^+})\to \tau^{\leq i}R\Gamma_{\rm zar}(\bar{\frak X}, \tilde \Omega_{\bar{\frak X}})
 \]
 whose composition in either direction is $(\zeta_p-1)^i$.
 
Note that there is an isomorphism $R\Gamma_{\rm zar}(\bar{\frak X}, R\nu_*{\hat {\cal O}}_{ X}^+)\simeq R\Gamma_{\rm pro{\acute e}t}(X, \hat{\cal O}_X^+)$. What we want to study at the end is not the pro-\'etale cohomology but the $p$-adic $\rm \acute e$tale cohomology. Actually we get almost what we want. Recall the primitive comparison theorem due to Scholze.
 \begin{theorem}[\hspace{1sp}{\cite[Theorem 8.4]{scholze2013p}}]\label{primitive}
 For any proper smooth adic space $Y$ over $C$, there are natural almost isomorphisms
 \[
 R\Gamma_{\rm {\acute e}t}(Y, \mathbb Z_p)\otimes_{\mathbb Z_p}^{\mathbb L}\cal O_C\simeq R\Gamma_{\rm pro{\acute e}t}(Y, \hat{\cal O}_Y^+)
\]
 and
\[
 R\Gamma_{\rm {\acute e}t}(Y, \mathbb Z_p)\otimes_{\mathbb Z_p}^{\mathbb L}A_{\inf}\simeq R\Gamma_{\rm pro{\acute e}t}(Y, \bb A_{\inf,Y}).
\]
 \end{theorem}
 
Then by passing to the world of almost mathematics, we get two natural maps in $D(\cal O_C)^a$:
 \[
 f^a:(\tau^{\leq i}R\Gamma_{\rm zar}(\bar{\frak X},  \tilde \Omega_{\bar{\frak X}}))^a\to (\tau^{\leq i}R\Gamma_{\rm zar}(\bar{\frak X}, R\nu_*{\hat {\cal O}}_{ X}^+))^a\simeq (\tau^{\leq i}R\Gamma_{\rm {\acute e}t}(X, \mathbb Z_p)\otimes_{\mathbb Z_p}\cal O_C)^a
 \]
 \[
 g^a:(\tau^{\leq i}R\Gamma_{\rm zar}(\bar{\frak X}, R\nu_*{\hat {\cal O}}_{ X}^+))^a\simeq (\tau^{\leq i}R\Gamma_{\rm {\acute e}t}(X, \mathbb Z_p)\otimes_{\mathbb Z_p}\cal O_C)^a \to (\tau^{\leq i}R\Gamma_{\rm zar}(\bar{\frak X}, \tau ^{\leq i}\tilde \Omega_{\bar{\frak X}}))^a
 \]

 \begin{lemma}\label{perfect}
 The complex $\tau^{\leq i}R\Gamma_{\rm HT}(\bar{\frak X})=\tau^{\leq i}R\Gamma_{\rm zar}(\bar{\frak X},  \tilde \Omega_{\bar{\frak X}})$ (resp. $\tau^{\leq i}R\Gamma_{A_{\inf}}(\bar{\frak X})$) is a perfect complex of $\cal O_C$-modules (resp. $A_{\inf}$-modules).
 \end{lemma}
 \begin{proof}
Recall that we have 
\[
R\Gamma_{\rm HT}(\bar {\frak X})\simeq R\Gamma_{\frak S}(\frak X)\otimes_{\frak S, \alpha}^{\bb L}A_{\inf}\otimes_{A_{\inf}}^{\bb L}A_{\inf}/{\tilde \xi}\simeq R\Gamma_{\frak S}(\frak X)\otimes_{\frak S,\beta}^{\bb L}\cal O_K\otimes_{\cal O_K}^{\bb L}\cal O_C.
\]
Since $R\Gamma_{\frak S}(\frak X)$ is a perfect complex of $\frak S$-modules and $R\Gamma_{\rm HT}(\frak X):=R\Gamma_{\frak S}(\frak X)\otimes_{\frak S,\beta}\cal O_K$, the Hodge--Tate cohomology $R\Gamma_{\rm HT}(\frak X)$ of $\frak X$ is a perfect complex of $\cal O_K$-modules by \cite[\href{https://stacks.math.columbia.edu/tag/066W}{Lemma 066W}]{stacks-project}. Moreover as $\cal O_K$ is a Notherian local ring, the cohomology groups $H^i_{\rm HT}(\frak X)$ are finitely generated $\cal O_K$-modules and so finitely presented $
\cal O_K$-modules. So we see that every Hodge--Tate cohomology group $H^i_{\rm HT}(\bar{\frak X})$ is also finitely presented over $\cal O_C$. By Lemma \ref{structure}, this means $H^i_{\rm HT}(\bar{\frak X})\cong\bigoplus_{j=1}^n\cal O_C/\pi_j$ for some $\pi_j\in \cal O_C$. So $H^i_{\rm HT}(\bar{\frak X})$ is perfect. The lemma hence follows from 
 \cite[\href{https://stacks.math.columbia.edu/tag/066U}{Lemma 066U}]{stacks-project}. 
 For $\tau^{\leq i}R\Gamma_{A_{\inf}}(\bar{\frak X})$, this follows from \cite[Lemma 4.9]{BMS1} stating that each $H^i_{A_{\inf}}(\bar{\frak X})$ is perfect.
 \end{proof}

 As $\tau^{\leq i}R\Gamma_{\rm{\acute e}t}(X, \mathbb Z_p)$ and $\tau^{\leq i}R\Gamma_{\rm zar}(\bar{\frak X},  \tilde \Omega_{\bar{\frak X}})$ are perfect complexes, Lemma \ref{almost1} shows that if $C$ is spherically complete, we then have $(\tau^{\leq i}R\Gamma(\bar{\frak X},  \tilde \Omega_{\bar{\frak X}}))^a_*\simeq \tau^{\leq i}R\Gamma(\frak X,  \tilde \Omega_{\frak X})$ and $(\tau^{\leq i}R\Gamma_{\rm{\acute e}t}(X, \mathbb Z_p)\otimes_{\mathbb Z_p}\cal O_C)^a_*\simeq \tau^{\leq i}R\Gamma_{\rm{\acute e}t}(X, \mathbb Z_p)\otimes_{\mathbb Z_p}\cal O_C$.
 By moving back to the real world, we have two maps
 \[
 (f^a)_*:\tau^{\leq i}R\Gamma_{\rm zar}(\bar{\frak X},  \tilde \Omega_{\bar{\frak X}}) \to  \tau^{\leq i}R\Gamma_{\rm{\acute e}t}(X, \mathbb Z_p)\otimes_{\mathbb Z_p}\cal O_C
 \]
 \[
 (g^a)_*: \tau^{\leq i}R\Gamma_{\rm{\acute e}t}(X, \mathbb Z_p)\otimes_{\mathbb Z_p}\cal O_C \to \tau^{\leq i}R\Gamma_{\rm zar}(\bar{\frak X}, \tilde \Omega_{\bar{\frak X}})
 \]
 whose composition in either direction is $(\zeta_p-1)^i$. These two maps induce maps between cohomology groups for any $n\leq i$.
 \[
 f: H^n(\bar{\frak X}, \tilde \Omega_{\bar{\frak X}})\to H^n_{\rm{\acute e}t}(X, \mathbb Z_p)\otimes_{\mathbb Z_p}\cal O_C
 \]
 \[
 g: H^n_{\rm{\acute e}t}(X, \mathbb Z_p)\otimes_{\mathbb Z_p}\cal O_C\to H^n(\bar{\frak X}, \tilde \Omega_{\bar{\frak X}})
 \]
 whose composition in either direction is  $(\zeta_p-1)^i$. 
 \\
 
Now we come to the following key theorem:
 \begin{theorem}\label{HT}
 Let $\frak X$ be a proper smooth formal scheme over $\cal O_K$, where $\cal O_K$ is the ring of integers in a complete discretely valued non-archimedean extension $K$ of $\bb Q_p$ with perfect residue field $k$ and ramification degree $e$. Let $\cal O_C$ be the ring of integers in a complete and algebraically closed extension $C$ of $K$ and $X$ be the adic generic fibre of $\bar{\frak X}:=\frak X\times_{{\rm Spf}(\cal O_K)}{\rm Spf}(\cal O_C)$. Assuming $ie<p-1$, there is an isomorphism of $\cal O_C$-modules between the Hodge--Tate cohomology group and the $p$-adic $\acute e$tale cohomology group
 \[
H^i_{\rm HT}(\bar{\frak X}):=H^i(\bar{\frak X}, \tilde \Omega_{\frak X})\cong H^i_{\rm{\acute e}t}(X, \mathbb Z_p)\otimes_{\mathbb Z_p}\cal O_C.
\]
\end{theorem}
\begin{proof}
Note that replacing $C$ by its spherical completion $C'$ will not make any difference to this theorem. The spherical completion always exists (cf. \cite[Chapter 3]{robert2013course}), which is still complete and algebraically closed. On one hand, $p$-adic $\rm \acute e$tale cohomology is insensitive to such extensions in the rigid-analytic setting (cf. \cite[Section 0.3.2]{huber2013etale}). On the other hand, by the base change of prismatic cohomology (cf. \cite[Theorem 1.8]{bhatt2019prisms}), we have $H^i_{\rm HT}(\overline{\frak X}\otimes_{\cal O_C}\cal O_{C'})\cong H^i_{\rm HT}(\overline{\frak X})\otimes_{\cal O_C}\cal O_{C'}$ and the natural injection $\cal O_C\to \cal O_{C'}$ is flat.

So now we assume $C$ is spherically complete. Using the flat base change along $\alpha:\frak S\to A_{\inf}$ from the Breuil--Kisin cohomology to the $A_{\inf}$-cohomology, we see that $H^i(\bar{\frak X}, \tilde \Omega_{\bar{\frak X}})$ has a decomposition as $\cal O_C^n\oplus(\bigoplus_{j=1}^{m}\cal O_C/{\pi^{m_j}})$.
By requiring $ie<p-1$, we have $v((\zeta_p-1)^i)<v(\pi)$ in $\cal O_C$ as $v((\zeta_p-1)^{p-1})=v(p)$ and $v(p)=v(\pi^e)$. Now the theorem follows from Lemma \ref{torsionlemma} and the existence of maps 
 \[
 f: H^i(\bar{\frak X}, \tilde \Omega_{\bar{\frak X}})\to H^i_{\rm{\acute e}t}(X, \mathbb Z_p)\otimes_{\mathbb Z_p}\cal O_C
 \]
 \[
 g: H^i_{\rm{\acute e}t}(X, \mathbb Z_p)\otimes_{\mathbb Z_p}\cal O_C\to H^i(\bar{\frak X}, \tilde \Omega_{\bar{\frak X}}).
 \]
\end{proof}

\section{The unramified case: comparison theorem}\label{D}
In this section, let $\cal O_K=W(k)$, i.e. the ramification degree $e=1$. We will study the relation between the $p$-adic \'etale cohomology group $H^i_{\rm {\acute e}t}(X, \bb Z_p)$ and the crystalline cohomology group $H^i_{\rm crys}(\frak X_k/W(k))$ in the unramifed case. Note that in the unramified case, the crystalline cohomology $R\Gamma_{\rm crys}(\frak X_k/W(k))$ is canonically isomorphic to the de Rham cohomology $R\Gamma_{\rm dR}(\frak X/\cal O_K)$.

In order to prove the integral comparison theorem, we first relate Hodge--Tate cohomology to Hodge cohomology. And then we can use Theorem \ref{HT} to get a link between Hodge cohomology and $p$-adic \'etale cohomology. The last step is to study the Hodge-to-de Rham spectral sequence and we can prove the converse to \cite[Theorem 14.5]{BMS1}, which results in the final comparison theorem.

\subsection{Decomposition of Hodge--Tate cohomology groups}
In this subsection, we explain how to relate Hodge--Tate cohomology to Hodge cohomology. In fact, we can show that the complex of sheaves $\tau^{\leq p-1}\tilde \Omega_{\bar{\frak X}}$ is formal in the unramified case.
\begin{theorem}\label{dht}
The complex of sheaves $\tau^{\leq p-1}\tilde \Omega_{\bar{\frak X}}$ is formal, i.e. there is an isomorphism 
\[
\gamma: \bigoplus_{i=0}^{p-1}\Omega_{\bar{\frak X}}^{i{}}\{-i\}[-i]\simeq \tau^{\leq p-1}\tilde \Omega_{\bar{\frak X}},
\]
where $\Omega^{i}_{\bar{\frak X}}:=\projlim \Omega^i_{(\bar{\frak X}/p^n)/(\cal O_C/p^n)}$ is the $\cal O_{\bar{\frak X}}$-module of continuous differentials and $\Omega_{\bar{\frak X}}^{i{}}\{-i\}$ is the Breuil--Kisin twist of $\Omega_{\bar{\frak X}}^{i{}}$.
\end{theorem}

\begin{proof}
We proceed by first showing that $\tau ^{\leq 1}\tilde \Omega_{\bar{\frak X}}$ is formal and then constructing the general isomorphism in the statement. In this proof, $\bb L_{\bar{\frak X}/\bb Z_p}$ and $\bb L_{\bar{\frak X}/W(k)}$ always mean the derived $p$-adic complete cotangent complex.




By \cite[Proposition 8.15]{BMS1}, there is an isomorphism $\tau ^{\leq 1}\tilde \Omega_{\bar{\frak X}}\simeq \mathbb L_{\bar{\frak X}/{\mathbb Z_p}}\{-1\}[-1]$. Considering the sequence of sheaves $\bb Z_p\to W(k)\to \cal O_{\bar{\frak X}}$, there is an associated distinguished triangle
\[
\hat{\bb L_{W(k)/\bb Z_p}}\otimes^{\bb L}_{W(k)}\cal O_{\bar{\frak X}}\to \bb L_{\bar {\frak X}/\bb Z_p}\to \bb L_{\bar{\frak X}/W(k)}.
\]
By derived Nakayama lemma, we know that $\hat{\bb L_{W(k)/\bb Z_p}}$ vanishes as $\bb L_{k/\bb F_p}$ vanishes. Therefore, we have
\[
{\mathbb L_{\bar{\frak X}/{\mathbb Z_p}}}\{-1\}[-1]\simeq {\bb L_{\bar{\frak X}/W(k)}}\{-1\}[-1].
\]

For any affine open ${\rm Spf}(R)\subset \frak X$, write $\bar R$ for the base change $R{\otimes}_{W(k)}\cal O_C$ and $\hat{\overline R}$ for its $p$-adic completion. Then we have $\hat{\bb L_{\hat{\overline R}/W(k)}}\simeq \hat{\bb L_{\bar R/W(k)}}$.

By the K$\rm \ddot u$nneth property of cotangent complex (cf. \cite{illusie2006complexe}), we get
\[
{\mathbb L_{\bar R/W(k)}}\simeq({\bb L_{\cal O_C/W(k)}}\otimes^{\bb L}_{W(k)}R)\oplus({\bb L_{R/W(k)}}\otimes^{\bb L}_{W(k)}\cal O_C).
\]

Applying the derived $p$-adic completion functor which is exact, we see
\[
\hat{\mathbb L_{R\otimes_{W(k)}\cal O_C/W(k)}}\simeq (\hat{{\bb L_{\cal O_C/W(k)}}\otimes^{\bb L}_{W(k)}R})\oplus(\hat{{\bb L_{R/W(k)}}\otimes^{\bb L}_{W(k)}\cal O_C}).
\]
On one hand, we have
\[
\hat{{\bb L_{\cal O_C/W(k)}}\otimes_{W(k)}R}\simeq \hat{\hat{{\bb L_{\cal O_C/W(k)}}}\otimes_{W(k)}R}\simeq \hat{{\bar R}\{1\}[1]}.
\]
As $\hat{\overline R}$ coincides with the derived $p$-adic completion of ${\bar R}$ (cf. \cite[\href{https://stacks.math.columbia.edu/tag/0BKG}{Example 0BKG}]{stacks-project}), we have $\hat{{\bb L_{\cal O_C/W(k)}}\otimes_{W(k)}R}\simeq \hat{\overline R}\{1\}[1]$. 
On the other hand, by the base change property of cotangent complex (cf. \cite{illusie2006complexe}), we get ${\bb L_{R/W(k)}}\otimes_{W(k)}\cal O_C\simeq \bb L_{\bar R/\cal O_C}$. The derived $p$-adic completion $\hat{\bb L_{\bar{R}/\cal O_C}}$ is isomorphic to $\projlim_n \Omega^1_{(\bar{R}/p^n)/(\cal O_C/p^n)}$. In fact as $\bar R/{p^n}$ is a smooth $\cal O_C/{p^n}$-algebra for all $n$, we have
\[
\hat{\bb L_{\bar{R}/\cal O_C}}\simeq {\rm Rlim}(\bb L_{\bar{R}/\cal O_C}\otimes^{\bb L}_{\bb Z_p}\bb Z_p/p^n)\simeq {\rm Rlim}(\bb L_{(\bar{R}/p^n)/(\cal O_C/p^n)})\simeq \projlim_n \Omega^1_{(\bar{R}/p^n)/(\cal O_C/p^n)}.
\]
So finally there is an isomorphism
\[
{\mathbb L_{\bar{\frak X}/W(k)}}\simeq \cal O_{\bar{\frak X}}\{1\}[1]\oplus \Omega^{1}_{\bar{\frak X}}
\]
and we get a decomposition $\tau ^{\leq 1}\tilde \Omega_{\bar{\frak X}}\simeq \cal O_{\bar{\frak X}}\oplus \Omega^{1}_{\bar{\frak X}}\{-1\}[-1]$. In particular, we have a map $\gamma_1: \Omega^{1}_{\bar{\frak X}}\{-1\}[-1]\to \tilde \Omega_{\bar{\frak X}}$ which gives the Hodge--Tate isomorphism $C^{-1}: \Omega_{\bar{\frak X}}^{1{}}\{-1\}\to \mathcal H^1(\tilde \Omega_{\bar{\frak X}})$ (cf. \cite[Theorem 8.3]{BMS1}).

Now we consider the map for any $i\leq p-1$ given by
\[
{(\Omega^{1}_{\bar{\frak X}})}^{\otimes i}\to \Omega^{i\,{}}_{\bar{\frak X}}, \ \ \  \omega_1\otimes \cdots \otimes \omega_i\mapsto \omega_1\wedge\cdots \wedge \omega_i
\]
It has an anti-symmetrization section $a$ as shown in \cite{deligne1987relevements}, given by
\[
a(\omega_1\wedge\cdots \wedge \omega_i)=(1/i!)\sum_{s\in Sym_i}{\rm sgn}(s)\omega_{s(1)}\otimes \cdots \otimes \omega_{s(i)}
\]
Then we define $\gamma_i$ as the composite
\[
\Omega^{i}_{\bar{\frak X}}\{-i\}[-i]\xrightarrow{a} {(\Omega^{1}_{\bar{\frak X}}\{-1\})}^{\otimes i}[-i]\simeq {(\Omega^{1}_{\bar{\frak X}}\{-1\}[-1])}^{\otimes^{\mathbb L} i}\xrightarrow{\gamma_1^{\otimes^{\mathbb L}i} }(\tilde \Omega_{\bar{\frak X}})^{\otimes^{\mathbb L}i} \xrightarrow{\rm multi}\tilde \Omega_{\bar{\frak X}}
\]
Note that $\tilde \Omega_{\bar{\frak X}}$ is a commutative $\cal O_{\overline{\frak X}}$-algebra object in $D(\cal O_{\overline{\frak X}})$ (see Remark \ref{commutative algebra structure}). By applying $\mathcal H^i$, we have
{\footnotesize{\[
{\Omega^{i}_{\bar{\frak X}}\{-i\}\xrightarrow{a}\mathcal H^i({(\Omega^{1}_{\bar{\frak X}}\{-1\}[-1])}^{\otimes^{\mathbb L}i}) \cong   (\mathcal H^1(\Omega^{1}_{\bar{\frak X}}\{-1\}[-1]))^{\otimes i}
\xrightarrow {\gamma_1^{\otimes^{\mathbb L}i}} (\mathcal H^1(\tilde \Omega_{\bar{\frak X}}))^{\otimes i}\to \mathcal H^i((\tilde \Omega_{\bar{\frak X}})^{\otimes^{\mathbb L}i} )\xrightarrow{\rm multi}\mathcal H^i(\tilde \Omega_{\bar{\frak X}})}
\]}}

Since the Hodge--Tate isomorphism is compatible with multiplication (cf. \cite[Corollary 8.13]{BMS1}), this composite is exactly the Hodge--Tate isomorphism $C^{-1}: \Omega^{i}_{\bar{\frak X}}\{-i\}\simeq \mathcal H^i(\tilde\Omega_{\bar{\frak X}})$.
So we get the desired isomorphism $\gamma=\bigoplus_{i=0}^{p-1}\gamma_i: \bigoplus_{i=0}^{p-1}\Omega_{\bar{\frak X}}^{i{}}\{-i\}[-i]\simeq \tau^{\leq p-1}\tilde \Omega_{\bar{\frak X}}$.
\end{proof}

\begin{remark}
Note that the key input in the proof above is the Hodge--Tate isomorphism $C^{-1}: \Omega_{\bar{\frak X}}^{i{}}\{-i\}\to \mathcal H^i(\tilde \Omega_{\bar{\frak X}})$. In general, there is a Hodge--Tate isomorphism for any bounded prism $(A, I)$ (cf. \cite[Theorem 4.10]{bhatt2019prisms}) and also a generalization of the isomorphism $\tau ^{\leq 1}\tilde \Omega_{\bar{\frak X}}\simeq {\mathbb L_{\bar{\frak X}/{\mathbb Z_p}}}\{-1\}[-1]$. 

The map $\cal O_{\bar{\frak X}}\to \tau^{\leq 1}\tilde \Omega_{\bar{\frak X}}$ splits as an $\cal O_{\bar{\frak X}}$-module map if and only if $\bar{\frak X}$ lifts to $A_{\inf}/{\tilde \xi}^2$ (cf. \cite[Remark 8.4]{BMS1}). In the ramified case, this seems to be hardly satisfied due to the non-vanishing of the cotangent complex $\bb L_{\cal O_K/W(k)}$. Note that $H^0(\bb L_{\cal O_K/W(k)})\simeq \Omega_{\cal O_K/W(k)}^1$ is generated by one element (cf. \cite[Chapter \uppercase\expandafter{\romannumeral3}, Proposition 14]{serre2013local}).

\end{remark}

\begin{corollary}\label{deligne}
There is a natural decomposition for any $n\leq p-1$,
\[
H^n_{\rm HT}(\bar{\frak X})=H^n(\bar{\frak X}, \tilde \Omega_{\bar{\frak X}})\cong \bigoplus_{i=0}^{n}H^{n-i}(\bar{\frak X}, \Omega_{\bar{\frak X}}^i\{-i\}).
\]
\end{corollary}

\subsection{Hodge-to-de Rham spectral sequence}

In this subsection, we study the Hodge-to-de Rham spectral sequence and finish the proof of the integral comparison theorem in the unramified case. More precisely, we will prove the converse to Theorem \ref{torsion} by analyzing the length of the torsion submodule of de Rham cohomology groups and that of $p$-adic \'etale cohomology groups.
\\

Note that we have the Hodge-to-de Rham spectral sequence
\[
E^{i,j}_1=H^j(\bar{\frak X}, \Omega^i_{\bar{\frak X}})\Longrightarrow H^{i+j}(\bar{\frak X}, \Omega^{\bullet}_{\bar{\frak X}})=H^{i+j}_{\rm dR}(\bar{\frak X}/\cal O_C)
\]
As $\bar{\frak X}=\frak X\times_{{\rm Spf}(W(k))}{\rm Spf}(\cal O_C)$, this spectral sequence can be seen as the flat base change to $\cal O_C$ of the Hodge-to-de Rham spectral sequence of $\frak X$ over $W(k)$. This tells us $E_{\infty}^{i,j}$ is a finitely presented $\cal O_C$-module (note that $E_{\infty}^{i,j}$ is also a subquotient of $H^j(\bar{\frak X}, \Omega^i_{\bar{\frak X}})$).

For any integers $i$ and $n$ such that $0\leq i\leq n$, we have the abutment filtration
\[
0=F^{n+1}\subset F^n\subset \cdots \subset F^0=H^{n}_{\rm dR}(\bar{\frak X}/\cal O_C)
\]
and the short exact sequences
\[
0\to F^{i+1}\to F^i\to E_{\infty}^{i, n-i}\to 0.
\]

Now we consider the normalized length $l_{\cal O_C}$ for finitely presented torsion $\cal O_C$-modules. Recall that this length behaves additively under short exact sequences and $l_{\cal O_C}(\cal O_C/p)=1$. For any finitely presented $\cal O_C$-module $M$, one can deduce from Lemma \ref{structure} that $M_{\rm tor}$ is also a finitely presented $\cal O_C$-module and so is $M_{\rm tor}/p^m$ for any $m>0$. Then we have the following lemma:

\begin{lemma}\label{torlength}
For any short exact sequence of finitely presented $\cal O_C$-modules
\[
 0\to A\to B \to C\to 0
 \]
 we have $ l_{\cal O_C}(B_{\rm tor})\leq l_{\cal O_C}(A_{\rm tor})+l_{\cal O_C}(C_{\rm tor})$ and $ l_{\cal O_C}(B_{\rm tor}/p^m)\leq l_{\cal O_C}(A_{\rm tor}/p^m)+l_{\cal O_C}(C_{\rm tor}/p^m)$ for any $m>0$.
\end{lemma}

\begin{proof}
For the first statement, it is easy to see that $M=B_{\rm tor}/{A_{\rm tor}}$ is a submodule of $C_{\rm tor}$, so we have $l_{\cal O_C}(M)=l_{\cal O_C}(B_{\rm tor})-l_{\cal O_C}(A_{\rm tor})\leq l_{\cal O_C}(C_{\rm tor})$ by the additivity of the length. For the second one,  we have an exact sequence
\[
M[p^m]\to A_{\rm tor}/p^m\to B_{\rm tor}/p^m\to M/p^m\to 0
\]
So we get $l_{\cal O_C}(B_{\rm tor}/p^m)\leq l_{\cal O_C}(A_{\rm tor}/p^m)+l_{\cal O_C}(M/p^m)$. Then we need to prove that $l_{\cal O_C}(M/p^m)\leq l_{\cal O_C}(C_{\rm tor}/{p^m})$. More generally, given two finitely presented torsion $\cal O_C$ modules $N_1\subset N_2$, there is an exact sequence
\[
N[p^m]\to N_1/p^m\to N_2/p^m\to N/p^m\to 0
\]
where $N=N_2/N_1$. Note that $l_{\cal O_C}(N[p^m])=l_{\cal O_C}(N/p^m)$. In fact, this follows from the exact sequence
\[
0\to N[p^m]\to N\xrightarrow{p^m}N\to N/p^m\to 0
\]
Hence $l_{\cal O_C}(N_2/p^m)\geq l_{\cal O_C}(N/p^m)+l_{\cal O_C}(N_1/{p^m})-l_{\cal O_C}(N[p^m])=l_{\cal O_C}(N_1/{p^m})$. So finally we get $ l_{\cal O_C}(B_{\rm tor}/p^m)\leq l_{\cal O_C}(A_{\rm tor}/p^m)+l_{\cal O_C}(C_{\rm tor}/p^m)$.
\end{proof}

\begin{corollary}\label{1st}
For any integers $i$ and $n$ such that $0\leq i\leq n$ and any positive integer $m$, we have $l_{\cal O_C}(F^i_{\rm tor}/p^m)\leq l_{\cal O_C}(F^{i+1}_{\rm tor}/p^m)+l_{\cal O_C}({E_{\infty}^{i, n-i}}_{\rm tor}/p^m)$. In particular, $l_{\cal O_C}(H^{n}_{\rm dR}(\bar{\frak X}/\cal O_C)_{\rm tor}/p^m)\leq \sum_{i=0}^{n}l_{\cal O_C}({E_{\infty}^{i, n-i}}_{\rm tor}/p^m)$.
\end{corollary}

Recall that the rational Hodge-to-de Rham spectral sequence degenerates at $E_1$ page:
\begin{theorem}[\hspace{1sp}{\cite[Corollary 1.8]{scholze2013p}}]\label{degeneration}
For any proper smooth rigid analytic space $Y$ over $C$, the Hodge-to-de Rham spectral sequence 
\[
E^{i,j}_1=H^j(Y, \Omega_Y^i)\Longrightarrow H^{i+j}_{\rm dR}(Y/C)
\]
degenerates at $E_1$. Moreover, for all $i\geq 0$, 
\[
\sum^i_{j=0}{\rm dim}_CH^{i-j}(Y, \Omega_Y^j)={\rm dim}_CH^i_{\rm dR}(Y/C)={\rm dim}_{\mathbb Q_p}H^i_{\rm {\acute et}}(Y, \mathbb Q_p).
\]

\end{theorem}

As a consequence, we have the following lemma:
\begin{lemma}\label{hh}
For any $m>0$, we have 
\[
l_{\cal O_C}({E_{\infty}^{i, n-i}}_{\rm tor}/p^m)\leq l_{\cal O_C}(H^{n-i}(\bar{\frak X}, \Omega_{\bar{\frak X}}^i)_{\rm tor}/p^m).
\]
\end{lemma}

\begin{proof}
Theorem \ref{degeneration} tells us that the integral Hodge-to-de Rham spectral sequence degenerates at $E_1$ after inverting $p$. This means that the coboundaries $B_{\infty}^{i,n-i}$ must be a finitely presented torsion $\cal O_C$-module. Consider the short exact sequence
\[
0\to B_{\infty}^{i,n-i}\to Z_{\infty}^{i,n-i}\to E_{\infty}^{i,n-i}\to 0.
\]
For any $x\in {E_{\infty}^{i,n-i}}_{\rm tor}$, there exists $\hat x\in Z_{\infty}^{i,n-i}$ whose image in $E_{\infty}^{i,n-i}$ is $x$. As $E_{\infty}^{i,n-i}$ is killed by $p^N$ for some large enough $N$, we can see that $p^N\hat x$ is in $B_{\infty}^{i,n-i}\subset {Z_{\infty}^{i,n-i}}_{\rm tor}$. So we have another short exact sequence
\[
0\to B_{\infty}^{i,n-i}\to {Z_{\infty}^{i,n-i}}_{\rm tor}\to {E_{\infty}^{i,n-i}}_{\rm tor}\to 0.
\]

Then by the additivity of the length, we get that
\[
l_{\cal O_C}({E_{\infty}^{i, n-i}}_{\rm tor}/p^m)\leq l_{\cal O_C}({Z_{\infty}^{i, n-i}}_{\rm tor}/p^m),
\]
and
{\footnotesize{\[
l_{\cal O_C}({Z_{\infty}^{i, n-i}}_{\rm tor}/p^m)=l_{\cal O_C}({Z_{\infty}^{i, n-i}}_{\rm tor}[p^m])\leq l_{\cal O_C}(H^{n-i}(\bar{\frak X}, \Omega_{\bar{\frak X}}^i)_{\rm tor}[p^m])=l_{\cal O_C}(H^{n-i}(\bar{\frak X}, \Omega_{\bar{\frak X}}^i)_{\rm tor}/p^m).
\]}}

So we have $l_{\cal O_C}({E_{\infty}^{i, n-i}}_{\rm tor}/p^m)\leq l_{\cal O_C}(H^{n-i}(\bar{\frak X}, \Omega_{\bar{\frak X}}^i)_{\rm tor}/p^m)$.
\end{proof}


Now we prove the converse to Theorem \ref{torsion}.
\begin{theorem}\label{unr}
For any positive integer $m$ and any integer $n$ such that $0\leq n<p-1$, we have 
\[
l_{\cal O_C}(H^{n}_{\rm dR}(\bar{\frak X}/\cal O_C)_{\rm tor}/p^m)\leq l_{\cal O_C}(H^n_{{\rm\acute {e}t}}(X, \mathbb Z_p)_{\rm tor}\otimes_{\mathbb Z_p}\cal O_C/p^m).
\]
\end{theorem}

\begin{proof}
By Theorem \ref{HT} and Theorem \ref{deligne}, we have
 \[
 H^n_{{\rm\acute {e}t}}(X, \mathbb Z_p)\otimes_{\mathbb Z_p}\cal O_C\cong H^n_{\rm HT}(\bar{\frak X})\cong\bigoplus_{i=0}^{n}H^{n-i}(\bar{\frak X}, \Omega^i_{\bar{\frak X}}).
 \]
 This implies that 
  \[
 \sum_{i=0}^{n}l_{\cal O_C}(H^{n-i}(\bar{\frak X}, \Omega_{\bar{\frak X}}^i)_{\rm tor}/p^m)=l_{\cal O_C}(H^n_{{\rm\acute {e}t}}(X, \mathbb Z_p)_{\rm tor}\otimes_{\mathbb Z_p}\cal O_C/p^m)
 \]
Moreover, by Corollary \ref{1st} and Lemma \ref{hh}, we have
\[
 l_{\cal O_C}(H^{n}_{\rm dR}(\bar{\frak X}/\cal O_C)_{\rm tor}/p^m)\leq \sum_{i=0}^{n}l_{\cal O_C}({E_{\infty}^{i, n-i}}_{\rm tor}/p^m)\leq \sum_{i=0}^{n}l_{\cal O_C}(H^{n-i}(\bar{\frak X}, \Omega_{\bar{\frak X}}^i)_{\rm tor}/p^m).
 \]
 So we get that
 \[
 l_{\cal O_C}(H^{n}_{\rm dR}(\bar{\frak X}/\cal O_C)_{\rm tor}/p^m)\leq l_{\cal O_C}(H^n_{{\rm\acute {e}t}}(X, \mathbb Z_p)_{\rm tor}\otimes_{\mathbb Z_p}\cal O_C/p^m)
 \]
 \end{proof}
 \begin{theorem}\label{unrcom}
For any $n<p-1$, there is an isomorphism of $W(k)$-modules
\[
H^n_{\rm crys}(\frak X_k/W(k))\cong H^n_{\rm\acute {e}t}(X, \mathbb Z_p)\otimes_{\mathbb Z_p}W(k).
\]
 \end{theorem}
 \begin{proof}
 We first prove that there is an isomorphism of $\cal O_C$-modules
 \[
 H^n_{\rm dR}(\bar{\frak X}/\cal O_C)\cong H^n_{\rm\acute {e}t}(X, \mathbb Z_p)\otimes_{\mathbb Z_p}\cal O_C.
 \]
Note that Theorem \ref{torsion} tells us that for any positive integer $m$,
 \[
 l_{\cal O_C}(H^n_{{\rm\acute {e}t}}(X, \mathbb Z_p)_{\rm tor}\otimes_{\mathbb Z_p}\cal O_C/p^m)\leq l_{\cal O_C}(H^{n}_{\rm dR}(\bar{\frak X}/\cal O_C)_{\rm tor}/p^m)
 \]
So they must be equal by Theorem \ref{unr}. This means that $H^n_{{\rm\acute {e}t}}(X, \mathbb Z_p)_{\rm tor}\otimes_{\mathbb Z_p}\cal O_C\cong H^{n}_{\rm dR}(\bar{\frak X}/\cal O_C)_{\rm tor}$ by Lemma \ref{lengthequal}. Furthermore by \cite[Theorem 1.1]{BMS1}, the $\cal O_C$-modules $H^{n}_{\rm dR}(\bar{\frak X}/\cal O_C)$ and $H^n_{{\rm\acute {e}t}}(X, \mathbb Z_p)\otimes_{\mathbb Z_p}\cal O_C$ have the same rank. So we have $H^{n}_{\rm dR}(\bar{\frak X}/\cal O_C)\cong H^n_{{\rm\acute {e}t}}(X, \mathbb Z_p)\otimes_{\mathbb Z_p}\cal O_C$. 

On the other hand, there is an isomorphism between de Rham cohomology and crystalline cohomology in the unramified case (cf. \cite{berthelot2006cohomologie})
\[
H^n_{\rm dR}(\frak X/W(k))\cong H^n_{\rm crys}(\frak X_k/W(k)).
\]
We also have 
\[
H^n_{\rm dR}(\frak X_{}/W(k))\otimes_{W(k)}\cal O_C\cong H^n_{\rm dR}(\bar{\frak X}/\cal O_C)
\]
by base change of de Rham cohomology. So finally we get the isomorphism of $W(k)$-modules
\[
H^n_{\rm crys}(\frak X_k/W(k))\cong H^n_{\rm\acute {e}t}(X, \mathbb Z_p)\otimes_{\mathbb Z_p}W(k).
\]
\end{proof}

\subsection{Degeneration of the Hodge-to-de Rham spectral sequence}
In this subsection, we assume $d={\rm dim}{\frak X}<p-1$, where ${\rm dim}{\frak X}$ means the relative dimension of $\frak X$. We will improve Theorem \ref{unrcom} by considering all cohomological degrees and study the degeneration of the Hodge-to-de Rham spectral sequence. These will follow from improvements of Theorem \ref{HT} and Corollary \ref{deligne}.

We begin with an improvement of Corollary \ref{deligne}.
\begin{lemma}\label{ght}
When $d={\rm dim}\frak X<p-1$, we have 
\[
H^n_{\rm HT}(\bar{\frak X})=H^n(\bar{\frak X}, \tilde \Omega_{\bar{\frak X}})\cong \bigoplus_{i=0}^{n}H^{n-i}(\bar{\frak X}, \Omega_{\bar{\frak X}}^i\{-i\}).
\]
for all $n$.
\end{lemma}
\begin{proof}
Recall the Hodge--Tate isomorphism: $H^i(\tilde \Omega_{\overline{\frak X}})\cong \Omega_{\overline{\frak X}}^{i{}}$ (cf. \cite[Theorem 8.3]{BMS1}). When $i\geq p-1>d$, we have $\Omega_{\overline{\frak X}}^{i{}}=0$. This implies $\tau^{\leq p-2}\tilde\Omega_{\overline{\frak X}}\simeq \tilde\Omega_{\overline{\frak X}}$. In particular, the whole complex $\tilde \Omega_{\overline{\frak X}}$ is formal by Theorem \ref{dht}, from which this lemma follows.

\end{proof}

Next we study the comparison between Hodge--Tate cohomology and $p$-adic \'etale cohomology. Recall that we have the following two maps
\[
f: \tau^{\leq d}\tilde\Omega_{\overline{\frak X}}\to \tau^{\leq d}R\nu_*\hat{\cal O}_X^+
\]
\[
g: \tau^{\leq d}R\nu_*\hat{\cal O}_X^+\to \tau^{\leq d}\tilde\Omega_{\overline{\frak X}}
\]
whose composition in either direction is $(\zeta_p-1)^{d}$.

We claim that $R\nu_*\hat{\cal O}_X^+$ is almost supported in degrees $\leq d$, i.e. there is an almost isomorphism $\tau^{\leq d}R\nu_*\hat{\cal O}_X^+\simeq R\nu_*\hat{\cal O}_X^+$. We will check this locally.

Recall that an $\cal O_C$-algebra $R$ is called formally smooth (as in \cite{BMS1}) if it is a $p$-adically complete flat $\cal O_C$-algebra such that $R/p$ is a smooth $\cal O_C/p$-algebra. And a formally smooth $\cal O_C$-algebra $R$ is called small (cf. \cite[Definition 8.5]{BMS1}) if there is an \'etale map 
\[
\Box: {\rm Spf}R\to {\rm Spf}\cal O_C\langle T_1^{\pm1},\cdots,T_d^{\pm1}\rangle.
\]
We call such \'etale map a framing. Given a framing, we can define 
\[
R_{\infty}:= R\hat{\otimes}_{\cal O_C\langle T_1^{\pm1},\cdots,T_d^{\pm1}\rangle}\cal O_C\langle T_1^{\pm1/{p^{\infty}}},\cdots,T_d^{\pm1/{p^{\infty}}}\rangle
\]
which is an integral perfectoid ring. And there is an action of $\Gamma=\bb Z_p(1)^d$ on it. More precisely, choose a compatible system $(\zeta_{p^m})$ of $p$-power roots of unity and let $\gamma_i$, $i=1,\cdots, d$ be generators of $\Gamma$. Then $\gamma_i$ acts by sending $T_i^{1/{p^m}}$ to $\zeta_{p^m}T_i^{1/{p^m}}$ and sending $T_j^{1/{p^m}}$ to $T_j^{1/{p^m}}$ for $j\neq i$.

By Faltings' almost purity theorem (cf. \cite[Chapter 1, Section 3 and 4]{faltings1988p}) and \cite[Proposition 3.5, Proposition 3.7, Corollary 6.6]{scholze2013p}, there is an almost isomorphism of complexes of $\cal O_C$-modules
\[
R\Gamma_{{}}(\Gamma, R_{\infty})\to R\Gamma(Y_{\rm{pro\acute et}}, \hat{\cal O}_Y^+),
\]
where $Y={\rm Spa}(R[1/p], R)$. Moreover the continuous group cohomology on the left hand side can be calculated by the Koszul complex $K_{R_{\infty}}(\gamma_1-1,\cdots, \gamma_d-1)$ by \cite[Lemma 7.3]{BMS1}, which can be defined as
\[
K_{R_{\infty}}(\gamma_1-1,\cdots, \gamma_d-1)=R_{\infty}\otimes_{\bb Z[\gamma_1,\cdots,\gamma_d]}(\bigotimes_{i=1}^d(\bb Z[\gamma_1,\cdots,\gamma_d]\xrightarrow{\gamma_i-1}\bb Z[\gamma_1,\cdots,\gamma_d])).
\]
This complex sits in non-negative cohomological degrees $[0, d]$. On the other hand, since $\overline{\frak X}$ is a proper smooth formal scheme over $\cal O_C$, there exists a basis of small affine opens (cf. \cite[Theorem 2]{kedlaya2003more}, \cite[Lemma 4.9]{bhatt2018specializing}). 
So when $i>d$, we get that $R^i\nu_*\hat{\cal O}_X^+$ is almost zero. 

So now we have an almost isomorphism: $\tau^{\leq d}R\nu_*\hat{\cal O}_X^+\to R\nu_*\hat{\cal O}_X^+$. Taking cohomology, we then get an almost isomorphism: $R\Gamma(\overline{\frak X}, \tau^{\leq d}R\nu_*\hat{\cal O}_X^+)\to R\Gamma(\overline{\frak X}, R\nu_*\hat{\cal O}_X^+)$. Again by Theorem \ref{primitive}, we get two maps in almost derived category $D(\cal O_C)^a$:
\[
f: (R\Gamma(\overline{\frak X}, \tau^{\leq d}\tilde{\Omega}_{\overline{\frak X}}))^a\to (R\Gamma_{\rm {\acute et}}(X, \bb Z_p)\otimes_{\bb Z_p}\cal O_C)^a
\]
\[
g: (R\Gamma_{\rm {\acute et}}(X, \bb Z_p)\otimes_{\bb Z_p}\cal O_C)^a\to (R\Gamma(\overline{\frak X}, \tau^{\leq d}\tilde{\Omega}_{\overline{\frak X}}))^a
\]
whose composition in either direction is $(\zeta_p-1)^{d}$.
Since both sides are perfect complexes of $\cal O_C$-modules, we get two maps in the derived category $D(\cal O_C)$:
\[
f: R\Gamma(\overline{\frak X}, \tau^{\leq d}\tilde{\Omega}_{\overline{\frak X}})\to R\Gamma_{\rm {\acute et}}(X, \bb Z_p)\otimes_{\bb Z_p}\cal O_C
\]
\[
g: R\Gamma_{\rm {\acute et}}(X, \bb Z_p)\otimes_{\bb Z_p}\cal O_C\to R\Gamma(\overline{\frak X}, \tau^{\leq d}\tilde{\Omega}_{\overline{\frak X}})
\]
whose composition in either direction is $(\zeta_p-1)^{d}$.

Now as $\tau^{\leq d}\tilde\Omega_{\overline{\frak X}}\simeq \tilde\Omega_{\overline{\frak X}}$, we have $R\Gamma(\overline{\frak X}, \tau^{\leq d}\tilde{\Omega}_{\overline{\frak X}})\simeq R\Gamma(\overline{\frak X},\tilde{\Omega}_{\overline{\frak X}})=R\Gamma_{\rm HT}(\overline{\frak X})$. So we get two maps
\[
f: R\Gamma_{\rm HT}(\overline{\frak X})\to R\Gamma_{\rm {\acute et}}(X, \bb Z_p)\otimes_{\bb Z_p}\cal O_C
\]
\[
g: R\Gamma_{\rm {\acute et}}(X, \bb Z_p)\otimes_{\bb Z_p}\cal O_C\to R\Gamma_{\rm HT}(\overline{\frak X})
\]
whose composition in either direction is $(\zeta_p-1)^{d}$.

\begin{theorem}\label{get}
There is an isomorphism of $\cal O_C$-modules for all $n$
\[
H^n_{\rm HT}(\overline{\frak X})\cong H^n_{\rm {\acute et}}(X, \bb Z_p)\otimes_{\bb Z_p}\cal O_C.
\]
\end{theorem}
\begin{proof}
This follows from Lemma \ref{torsionlemma}.
\end{proof}

\begin{theorem}\label{dimension}
Assume $d={\rm dim}\frak X<p-1$. Then there is an isomorphism of $W(k)$-modules for all $n$
\[
H^n_{\rm crys}(\frak X_k/W(k))\cong H^n_{\rm\acute {e}t}(X, \mathbb Z_p)\otimes_{\mathbb Z_p}W(k).
\]
\end{theorem}

\begin{proof}
Note that if Theorem \ref{unr} is true for all $n$, then Theorem \ref{unrcom} is true for all $n$. And if Theorem \ref{HT} and Theorem \ref{deligne} are true for all cohomological degrees, then Theorem \ref{unrcom} is true for all cohomological degrees. So this theorem follows from Theorem \ref{ght} and Theorem \ref{get}.

\end{proof}

\begin{corollary}\label{de}
If $d={\rm dim}(\frak X)<p-1$, the coboundaries $B^{i,n-i}_{\infty}$ vanish for all $n$. In particular the Hodge-to-de Rham spectral sequence degenerates at $E_1$-page.
\end{corollary}
\begin{proof}
By Theorem \ref{ght} and Theorem \ref{get}, we see that
  \[
 \sum_{i=0}^{n}l_{\cal O_C}(H^{n-i}(\bar{\frak X}, \Omega_{\bar{\frak X}}^i)_{\rm tor}/p^m)=l_{\cal O_C}(H^n_{{\rm\acute {e}t}}(X, \mathbb Z_p)_{\rm tor}\otimes_{\mathbb Z_p}\cal O_C/p^m)
 \]
is true for all $n$.

Theorem \ref{dimension} shows that for all $n$ we have
\[
l_{\cal O_C}(H^{n}_{\rm dR}(\bar{\frak X}/\cal O_C)_{\rm tor}/p^m)=l_{\cal O_C}(H^n_{{\rm\acute {e}t}}(X, \mathbb Z_p)_{\rm tor}\otimes_{\mathbb Z_p}\cal O_C/p^m).
\]

So we conclude that 
\[
l_{\cal O_C}(H^{n}_{\rm dR}(\bar{\frak X}/\cal O_C)_{\rm tor}/p^m)=\sum_{i=0}^{n}l_{\cal O_C}(H^{n-i}(\bar{\frak X}, \Omega_{\bar{\frak X}}^i)_{\rm tor}/p^m)
\]
holds for all $n$. 

As we have seen in the proof of Lemma \ref{hh}, there are inequalities for all $n$
\[
l_{\cal O_C}({E_{\infty}^{i, n-i}}_{\rm tor}/p^m)\leq l_{\cal O_C}({Z_{\infty}^{i, n-i}}_{\rm tor}/p^m)\leq l_{\cal O_C}(H^{n-i}(\bar{\frak X}, \Omega_{\bar{\frak X}}^i)_{\rm tor}/p^m).
\]

Also by using the same argument as in the proof of Theorem \ref{unr}, we have 
\[
 l_{\cal O_C}(H^{n}_{\rm dR}(\bar{\frak X}/\cal O_C)_{\rm tor}/p^m)\leq \sum_{i=0}^{n}l_{\cal O_C}({E_{\infty}^{i, n-i}}_{\rm tor}/p^m)\leq \sum_{i=0}^{n}l_{\cal O_C}(H^{n-i}(\bar{\frak X}, \Omega_{\bar{\frak X}}^i)_{\rm tor}/p^m).
 \]
 holds for all $n$. But these inequalities are in fact equalities. This means that 
\[
l_{\cal O_C}({E_{\infty}^{i, n-i}}_{\rm tor}/p^m)=l_{\cal O_C}({Z_{\infty}^{i, n-i}}_{\rm tor}/p^m)=l_{\cal O_C}(H^{n-i}(\bar{\frak X}, \Omega_{\bar{\frak X}}^i)_{\rm tor}/p^m).
\]
In other words, the coboundaries $B^{i,n-i}_{\infty}$ vanish as we have $l_{\cal O_C}(B_{\infty}^{i,n-i})=l_{\cal O_C}({Z^{i,n-i}_{\infty}}_{\rm tor})-l_{\cal O_C}({E^{i,n-i}_{\infty}}_{\rm tor})=0$. So the Hodge-to-de Rham spectral sequence degenerates at $E_1$-page.
\end{proof}

\begin{remark}\label{degeneration}
We collect some other results about the degeneration of the (integral) Hodge-to-de Rham spectral sequence.

\begin{enumerate}
\item In \cite[Corollary 2.7]{Fontaine--Messing}, Fontaine and Messing have proved that for any proper smooth (formal) scheme $\frak Y$ whose special fiber has dimension strictly less than $p$, the Hodge-to-de Rham spectral sequence degenerates at $E_1$-page. Their proof makes use of the syntomic cohomology.
\item
For any projective smooth scheme $\frak Y$ over $W(k)$ where $k$ is a perfect field of characteristic $p$, Kazuya Kato has proved that if ${\rm dim}(\frak Y)\leq p$, the Hodge-to-de Rham spectral sequence degenerates at $E_1$-page and the de Rham cohomology groups are Fontane--Laffaille modules (cf. \cite[chapter \uppercase\expandafter{\romannumeral2}, Proposition 2.5]{kato1987p}).
\item 
For any proper smooth formal scheme $\frak Y$ over $\cal O_K$, where $\cal O_K$ is the ring of integers of a complete discretely valued non-archimedean extension $K$ of $\bb Q_p$ with perfect residue field $k$ and ramification degree $e$. Let $\frak S$ be $W(k)[[u]]$ and $E$ be an Eisenstein polynomial for a uniformizer $\pi$ of $\cal O_K$. Shizhang Li has proved that if $\frak Y$ can be lifted to $\frak S/(E^2)$ and ${\rm dim}(\frak Y)\cdot e<p-1$, then the Hodge-to-de Rham spectral sequence is split degenerate (cf. \cite[Theorem 1.1]{li2020integral}). His proof uses Theorem \ref{main}.
\end{enumerate}

\end{remark}

\section{The ramified case: comparison theorem}\label{E}
In this section, we will get some properties about the torsion in the Breuil--Kisin cohomology groups $H^{i+1}_{\frak S}(\frak X)$ when $ie<p-1$ and obtain an integral comparison theorem comparing the $p$-adic \'etale cohomology groups and the crystalline cohomology groups.
~\\

\subsection{Torsion in Breuil--Kisin cohomology groups}
Note that the ring $\frak S=W(k)[[u]]$ is a two-dimensional regular local ring. The structure of $\frak S$-modules is subtle in general (see Remark \ref{torsionstructure}). In particular, it is difficult to study the $u$-torsion. But in our case, it turns out to be simpler. 

Recall that we can define $A_{\inf}:=W(\cal O_C^{\flat})$ as in Definition \ref{mu}. We start by studying the $A_{\inf}$-cohomology groups  of $\overline{\frak X}$.

\begin{lemma}\label{torsionfree}
The $A_{\inf}$-cohomology group $H^{i+1}_{A_{\inf}}(\bar{\frak X}):=H^{i+1}(\bar{\frak X}, A\Omega_{\bar{\frak X}})$ is $\tilde \xi$-torsion-free for any $i$ such that $ie<p-1$. 
\end{lemma}
\begin{proof}
We assume that $C$ is spherically complete. As in the proof of Theorem \ref{HT}, we see that the spherical completion of $C$ exists and is still complete and algebraically closed. Moreover since $R\Gamma_{A_{\inf}}(\overline{\frak X})\simeq R\Gamma_{\frak S}(\frak X)\otimes_{\frak S, \alpha}^{\bb L}A_{\inf}$ where $\alpha: \frak S\to A_{\inf}$ is the faithfully flat map taking $(E)$ to $(\tilde \xi)$, we have $H^{i+1}_{A_{\inf}}(\bar{\frak X})\cong H^{i+1}_{\frak S}(\frak X)\otimes_{\frak S, \alpha}A_{\inf}$, in particular $H^{i+1}_{A_{\inf}}(\bar{\frak X})$ is $\tilde \xi$-torsion-free if and only if $H^{i+1}_{\frak S}(\frak X)$ is $E$-torsion-free as $(\alpha(E))=(\tilde\xi)$. So it does not matter whether $C$ is spherically complete or not.

As in Chapter 3, we apply Lemma \ref{bhattkey} to the complex of sheaves of $A_{\inf}$-modules $\tau^{\leq i}R\nu_{*} \bb A_{\inf,X}$ and the element $\mu\in A_{\inf}$. Precisely, in the category $D^{[0,i]}(\overline{\frak X}, A_{\inf})$, we get two natural maps
\[
f: \tau^{\leq i}R\nu_{*} \bb A_{\inf,X}\to L\eta_{\mu}\tau^{\leq i}R\nu_{*} \bb A_{\inf,X}\simeq \tau^{\leq i}A\Omega_{\overline{\frak X}}
\]
\[
g: \tau^{\leq i}A\Omega_{\overline{\frak X}}\simeq L\eta_{\mu}\tau^{\leq i}R\nu_{*} \bb A_{\inf,X}\to \tau^{\leq i}R\nu_{*} \bb A_{\inf,X}
\]
whose composition in either direction is $\mu^i$.

We consider the the complex of sheaves $\tau^{\leq i}R{\nu}_*\widehat {\mathcal O}_X^+$ as in the category $D(\bar{\frak X}, A_{\rm inf})$ via the map $A_{\rm inf}\xrightarrow{\tilde \theta} \mathcal O_C\to \mathcal O_{\bar{\frak X}}$. Moreover it is in the category $D^{[0, i]}(\bar {\frak X}, A_{\rm inf})$. 

There is a map $\tau^{\leq i}R\nu_*\mathbb A_{\rm inf,X}\to \tau^{\leq i}R{\nu}_*\widehat {\mathcal O}_X^+$ induced by $\tilde \theta: \mathbb A_{\rm inf,X}\to \widehat{\mathcal O}_X^+$. So we can get a commutative diagram
\[
\begin{tikzcd}
L\eta_{\mu}\tau^{\leq i}R\nu_*\mathbb A_{\inf,X}\arrow{r}{s_1}\arrow[xshift=3ex]{d}{f_1} & L\eta_{\mu}\tau^{\leq i}R{\nu}_*\widehat {\mathcal O}_X^+\arrow[xshift=3ex]{d}{f_2}\\
\tau^{\leq i}R\nu_*\mathbb A_{\inf,X}\arrow{r}{s_2} \arrow{u}{g_1}& \tau^{\leq i}R{\nu}_*\widehat {\mathcal O}_X^+\arrow{u}{g_2}
\end{tikzcd}
\]
where the composition of $f_j$ with $g_j$ in either direction is $\mu^i$ for $j=1, 2$. Note that  $L\eta_{\zeta_p-1}\tau^{\leq i}R\nu_*\widehat{\mathcal O}_X^+$ is isomorphic to $L\eta_{\mu}\tau^{\leq i}R{\nu}_*\widehat {\mathcal O}_X^+$ in $D(\bar{\mathfrak X}, A_{\inf})$.

Recall that $\tau^{\leq i}R\Gamma_{A_{\inf}}(\bar{\frak X})$ is a perfect complex of $A_{\inf}$-modules according to Lemma \ref{perfect}. Then by the second almost isomorphism in Theorem \ref{primitive} and Lemma \ref{almostainf}, we can get two maps
\[
f: \tau^{\leq i}R\Gamma_{A_{\inf}}(\bar{\frak X})\to \tau^{\leq i}R\Gamma_{\rm\acute {e}t}(X, \mathbb Z_p)\otimes_{\mathbb Z_p}A_{\inf}.
\]
\[
g: \tau^{\leq i}R\Gamma_{\rm\acute {e}t}(X, \mathbb Z_p)\otimes_{\mathbb Z_p}A_{\inf}\to \tau^{\leq i}R\Gamma_{A_{\inf}}(\bar{\frak X})
\]
whose composition in either direction is $\mu^i$.

By taking cohomology, we can obtain another commutative diagram
\[
\begin{tikzcd}
H^i_{A_{\inf}}(\bar{\mathfrak X})\arrow{r}{s_1}\arrow[xshift=2.5ex]{d}{f_1} & H^i_{\rm HT}(\bar{ \mathfrak X})\arrow[xshift=2.5ex]{d}{f_2}\\
H^i_{\rm \acute et}(X, \mathbb Z_p)\otimes_{\mathbb Z_p}A_{\rm inf}\arrow{r}{s_2} \arrow{u}{g_1} & H^i_{\rm \acute et}(X, \mathbb Z_p)\otimes_{\mathbb Z_p}\mathcal O_C\arrow{u}{g_2}
\end{tikzcd}
\]
Note that ${\rm Coker}(s_1)$ is in fact $H^{i+1}_{A_{\inf}}(\bar{\mathfrak X})[\tilde\xi]$ and ${\rm Coker}(s_2)=0$. 

Therefore we get two induced maps
\[
 \begin{tikzcd} 
H^{i+1}_{A_{\inf}}(\bar{\mathfrak X})[\tilde\xi]\arrow[yshift=1ex]{r}{f_3} & 0\arrow{l}{g_3}
\end{tikzcd}
\]
where the composition of $f_3$ and $g_3$ in either direction is $\mu^i$. Since $H^{i+1}_{A_{\inf}}(\bar{\frak X})[{\tilde \xi}]\simeq H^{i+1}_{\frak S}(\frak X_{})[E]\otimes_{\mathcal O_K} \mathcal O_C$ as $\mathcal O_C$-modules, it has a decomposition as $\cal O_C^m\oplus(\bigoplus_{s=1}^{n}\mathcal O_C/{\pi^{n_s}})$. Note that the image of $\mu$ under the reduction $A_{\inf}\to A_{\inf}/{\tilde \xi}$ is $\zeta_p-1$ and $v((\zeta_p-1)^i)<v(\pi)$ when $ie<p-1$. We then can get $H^{i+1}_{A_{\inf}}(\bar{\frak X})[{\tilde \xi}]=0$ by Lemma \ref{torsionlemma}.

\end{proof}

\begin{remark}
The previous version of this lemma covers the cohomological degree $i$ such that $ie<p-1$. We want to thank Shizhang Li for pointing out that the previous proof can be improved slightly to include the cohomological degree $i+1$ such that $ie<p-1$.

\end{remark}

In the next lemma, we give an equivalent statement to the $\tilde \xi$-torsion-freeness for some special $A_{\inf}$-modules.
\begin{lemma}\label{max}
Let $M$ be a finitely presented $A_{\inf}$-module such that $M[\frac{1}{p}]$ is finite projective over $A_{\inf}[\frac{1}{p}]$, and let $x\in \frak m\backslash (p)$ where $\frak m$ is the maximal ideal of $A_{\inf}$. Then $M$ is $\tilde \xi$-torsion-free if and only if it is $x$-torison-free.
\end{lemma}
\begin{proof}
Note that the radical ideal of $(p,x)$ is the maximal ideal. If there exists $a\in M$ such that $xa=0$, then for any other $y\in m\backslash (p)$, we have $y^na=0$ for any sufficiently large $n$. This is because all torsion in $M$ is killed by some power of $p$. Then this lemma follows.
\end{proof}
\begin{corollary}\label{bktor}
When $ie<p-1$, the $A_{\inf}$-cohomology group $H^{i+1}_{A_{\inf}}(\bar{\frak X})$ is $\xi$-torsion-free and the Breuil--Kisin cohomology group $H^{i+1}_{\frak S}(\frak X)$ is both $E$-torsion-free and $u$-torsion-free.
\end{corollary}

Recall that for any finitely presented $A_{\inf}$-module $M$ such that $M[\frac{1}{p}]$ is finite projective over $A_{\inf}[\frac{1}{p}]$, we have the following proposition:
\begin{proposition}[\cite{BMS1} Proposition 4.13]\label{413}
Let $M$ be a finitely presented $A_{\inf}$-module such that $M[\frac{1}{p}]$ is finite projective over $A_{\inf}[\frac{1}{p}]$. Then there is a functorial exact sequence
\[
0\to M_{\rm tor}\to M\to M_{\rm free}\to \overline M\to 0
\]
satisfying:
\begin{enumerate}
\item $M_{\rm tor}$, the torsion submodule of $M$, is finitely presented and perfect as an $A_{\inf}$-module, and is killed by $p^n$ for $n\gg 0$.
\item $M_{\rm free}$ is a finite free $A_{\inf}$-module.
\item $\overline M$ is finitely presented and perfect as an $A_{\inf}$-module, and is supported at the closed point $s\in {\rm Spec}(A_{\inf})$.
\end{enumerate}

\end{proposition}
Here we recall the construction of the free module $M_{\rm free}$. Since $M/M_{\rm tor}$ is torsion-free, the quasi-coherent sheaf associated to it restricts to a vector bundle on ${\rm Spec}(A_{\inf})\backslash\{s\}$ by \cite[Lemma 4.10]{BMS1}. By \cite[Lemma 4.6]{BMS1}, the global section of this vector bundle is a finite free $A_{\inf}$-module, which gives $M_{\rm free}$. In particular, if $M/M_{\rm tor}$ is free itself, then $M/M_{\rm tor}=M_{\rm free}$. For more details, see the proof of \cite[Proposition 4.13]{BMS1}.

By applying this proposition to $H^{i}_{A_{\inf}}(\bar{\frak X})$, we can obtain the following lemma saying that $H^{i}_{A_{\inf}}(\bar{\frak X})$ is a direct sum of its torsion submodule and a free $A_{\inf}$-module.
\begin{lemma}\label{vanish}
For any $i$ such that $ie<p-1$, the term $\overline M$ in the functorial exact sequence 
\[
0\to M_{\rm tor}\to M=H^{i}_{A_{\inf}}(\bar{\frak X})\to M_{\rm free}\to \overline M\to 0
\]
vanishes.
\end{lemma}

\begin{proof}
Let $N=H^i_{\rm\acute {e}t}(X, \mathbb Z_p)\otimes_{\mathbb Z_p}A_{\inf}$, we have two maps $f: M\to N$ and $g:N\to M$, whose composition in either direction is $\mu^i$. Then we have a commutative diagram
\[
\begin{tikzcd}
  0 \arrow[r] & M_{\rm tor} \arrow[d]{f_1} \arrow[r] & M \arrow[d]{f_2} \arrow[r] & M_{\rm free} \arrow[d]{f_3} \arrow[r] & \overline M  \arrow[d]{f_4} \arrow[r] &0 \\
  0 \arrow[r] & N_{\rm tor}\arrow[r] & N\arrow[r] & N_{\rm free}\arrow[r] & 0\arrow[r] &0
\end{tikzcd}
\]
by functoriality. All the vertical maps have inverses up to $\mu^i$.

On the other hand, the exact sequence associated to $H^i_{A_{\inf}}(\bar{\frak X})$ is the flat base change of the canonical exact sequence associated to $H^{i}_{\frak S}(\frak X)$ (see \cite[Proposition 4.3 and 4.13]{BMS1}). Hence $\overline M\cong \overline {H^{i}_{\frak S}(\frak X_{})}\otimes_{\frak S}A_{\inf}$ and $\overline M/{\tilde \xi}\cong (\overline {H^{i}_{\frak S}(\frak X_{})}/E)\otimes_{\frak S}A_{\inf}$ where $\overline {H^{i}_{\frak S}(\frak X_{})}$ is a torsion $\frak S$-module and is killed by some power of $(p, u)$. Again, by using the decomposition of $\overline{H^{i}_{\frak S}(\frak X_{})}/E$ and the fact that $v((\zeta_p-1)^i)<v(\pi)$ when $ie<p-1$, we get $\overline{H^{i}_{\frak S}(\frak X_{})}/E=0$ and $\overline M/{\tilde \xi}=0$ by Lemma \ref{torsionlemma}. Then $\overline M=0$ follows from Nakayama lemma.

\end{proof}
\begin{corollary}\label{finalcor}
For any $i$ such that $ie<p-1$, the $A_{\inf}$-cohomology group $H^i_{A_{\inf}}(\overline{\frak X})$ is a direct sum of a free $A_{\inf}$-module and its torsion submodule. Also, the Breuil--Kisin cohomology group $H^i_{\frak S}(\frak X)$ is a direct sum of a free $\frak S$-module and its torsion submodule.
\end{corollary}

In the following part, we consider the torsion submodules of the cohomology groups $H^i_{A_{\inf}}(\bar{\frak X})$ and $H^i_{\frak S}(\frak X_{})$, and let ${H^i_A}_{-\rm tor}$, ${H^i_{\frak S}}_{-\rm tor}$ denote them respectively.



We first prove a key lemma which enables us to study the structure of ${H^n_{\frak S}}_{-\rm tor}$.

\begin{lemma}\label{key}
For any $i$ such that $ie<p-1$, the modules ${(p^sH^i_{A_{}}}_{-\rm tor})/p^{m}$ (resp. \\$(p^sH^i_{\frak S-\rm tor})/p^m)$ are $\tilde \xi$-torsion-free (resp. $E$-torsion-free) for all non-negative integers $m, s$. 
\end{lemma}

\begin{proof}
Recall that we have two injective map $f: {H^i_A}_{-\rm tor}\to { H^i_{\rm\acute {e}t}}_{-\rm tor}\otimes_{\mathbb Z_p}A_{\inf}$ and $g: { H^i_{\rm\acute {e}t}}_{-\rm tor}\otimes_{\mathbb Z_p}A_{\inf}\to {H^i_A}_{-\rm tor}$ whose composition in either direction is $\mu^i$. These induce two new maps (we still denote $f$ and $g$) between $({(p^sH^i_{A_{}}}_{-\rm tor})/p^{m})[\tilde \xi]$ and $({(p^sH^i_{\rm\acute {e}t}}_{-\rm tor})\otimes_{\mathbb Z_p}A_{\inf}/{p^m})[\tilde \xi]$ whose composition in either direction is $\mu^i$. Note that $((p^s{H^i_{\rm\acute {e}t}}_{-\rm tor})\otimes_{\mathbb Z_p}A_{\inf}/{p^m})[\tilde \xi]=0$. This means $((p^s{H^i_A}_{-\rm tor})/{p^m})[\tilde \xi]$ is killed by $\mu^i$. As $((p^s{H^i_A}_{-\rm tor})/{p^m})[\tilde \xi]\cong ((p^s{H^i_{\frak S}}_{-\rm tor})/{p^m})[E]\otimes_{\frak S}A_{\inf}$ admits a decomposition as $\bigoplus_{t=1}^{n}\mathcal O_C/{\pi^{n_t}}$ and $v((\zeta_p-1)^i)<v(\pi)$, the module $((p^s{H^i_A}_{-\rm tor})/{p^m})[\tilde \xi]$ must be 0 by Lemma \ref{torsionlemma}. Since $((p^s{H^i_A}_{-\rm tor})/{p^m})[\tilde \xi]\cong (p^sH^i_{\frak S-\rm tor})/p^m)[E]\otimes_{\frak S,\alpha}A_{\inf}$ and the map $\alpha: \frak S\to A_{\inf}$ is faithfully flat, we also have $(p^sH^i_{\frak S-\rm tor})/p^m)$ is $E$-torsion-free.
\end{proof}

In order to determine the module structure of $H^i_{\frak S}(\frak X)$, we need the following lemma.
\begin{lemma}\label{2mod}
Let $M$ be a finitely presented torsion $\frak S$-module. If $M/p\cong (\frak S/p)^n$ and $pM\cong \bigoplus_{i=1}^r \frak S/p^{n_i}$, we have an isomorphism of $\frak S$-modules: $M\cong \bigoplus_{i=1}^{n}\frak S/p^{m_i}$.
\end{lemma}
\begin{proof}
The proof is just that of \cite[Lemma 2.3.1.1]{breuil1998construction}, simply by replacing $S$ by $\frak S$. For readers' convenience, we give the proof here.

Choose $m\geq 0$ such that $p^mM=0$. Let $(e_1, e_2,\cdots, e_n)$ be a basis of $M/pM$ over $\frak S/p$ and we choose their liftings $\hat e_1, \hat e_2,\cdots, \hat e_n$ in $M$. By Nakayama lemma, we see that $M$ is generated by $(\hat e_1, \hat e_2,\cdots, \hat e_n)$ as a $\frak S/p^m$-module. So $(p\hat e_1, p\hat e_2, \cdots, p\hat e_n)$ generate the $\frak S/p^m$-module $pM$. 

After renumbering $(\hat e_i)$, we can suppose that the images of $p\hat e_1, p\hat e_2,\cdots, p\hat e_r$ in $pM\otimes_{\frak S/p^m}k$ form a basis over $k$. Choose $f_1,\cdots, f_r\in pM$ such that $pM\cong \bigoplus_{i=1}^{r}\frak S/p^{n_i}\frak S\cdot f_i$. Then there exists a $r\times r$-matrix $A\in M_r(\frak S/p^m\frak S)$ such that $(f_1, f_2, \cdots, f_r)A=(p\hat e_1, p\hat e_2, \cdots, p\hat e_r)$. Since $A\mod(p, u)\in GL_r(k)$, we know that $A$ is in $GL_r(\frak S/p^m\frak S)$. So we can replace $(\hat e_1, \hat e_2, \cdots, \hat e_r)$ by $(\hat e_1, \hat e_2, \cdots, \hat e_r)A^{-1}$ and suppose $p\hat e_i=f_i$ for $1\leq i\leq r$.

For $r+1\leq j\leq n$, there exist $a_{ij}\in \frak S/p^m\frak S$ for $1\leq i\leq r$ such that $p\hat e_j=\sum_{i=1}^ra_{ij}f_i=\sum_{i=1}^{r}a_{ij}p\hat e_i$. Again, we can replace $\hat e_j$ by $\hat e_j-\sum_{i=1}^ra_{ij}\hat e_i$ for $r+1\leq j\leq n$. That means we can suppose $p\hat e_j=0$ for $r+1\leq j\leq n$. 

Finally, we can construct a surjective morphism of $\frak S/p^m\frak S$-module:
\[
h: M^{\prime}=(\bigoplus_{i=1}^r\frak S/p^{n_i+1}\frak S\times g_i)\bigoplus (\bigoplus_{i=r+1}^n\frak S/p\frak S\times g_i)\to M
\]
\[
g_i\mapsto\hat e_i
\]
Note that the morphism $h:M^{\prime}\to M$ induces two isomorphisms: $h_1: pM^{\prime}\xrightarrow{\sim}pM$ and $h_2: M^{\prime}/pM^{\prime}\xrightarrow{\sim}M/pM$ under the choice of $\hat e_i$, $1\leq i\leq n$. For any $x$ such that $h(x)=0$, if $x\in pM^{\prime}$, then $x=0$ since $h_2(x)=h(x)=0$. If $x\notin pM^{\prime}$, then $h_2(\bar x)=0$ implies that $x\in pM^{\prime}$ where $\bar x$ is the image of $x$ in $M^{\prime}/pM^{\prime}$. So $h: M^{\prime}\to M$ must be an isomorphism. We are done.
\end{proof}

\begin{corollary}\label{key2}
Let $M$ be a finitely presented torsion $\frak S$-module which is killed by some power of $p$. If $(p^sM)/p$ is $u$-torsion-free for all $s\geq 0$, the module $M$ admits a decomposition as $M\cong \bigoplus_{i=1}^{n}\frak S/p^{m_i}$.
\end{corollary}

\begin{proof}
To prove this corollary, we want to apply Lemma \ref{2mod} to $M$. Note that $M/p$ is $u$-torsion-free by our assumption, therefore finite free as a $\frak S/p=k[[u]]$-module. So we need to prove that $pM$ admits a nice decomposition as in Lemma \ref{2mod}. Since the module $(pM)/p$ is also $u$-torsion-free by our assumption, we only need to prove that $p^2M$ admits a nice decomposition as in Lemma \ref{2mod}. We can continue this process until that we need to prove $p^mM$ admits a nice decomposition as in Lemma \ref{2mod} for some $m$ such that $M$ is killed by $p^{m+1}$. As $p(p^mM)=0$ and $(p^mM)/p=p^mM$ has no $u$-torsion, we see that $p^mM$ is a free $\frak S/p$-module by Lemma \ref{2mod}. So we are done.

\end{proof}

\subsection{Integral comparison theorem}

Now we state our main theorem of this section comparing the module structure of Breuil--Kisin cohomology groups to that of $p$-adic \'etale cohomology groups.
\begin{theorem}\label{comp}
Let $\frak X$ be a proper smooth formal scheme over $\cal O_K$, where $\cal O_K$ is the ring of integers in a complete discretely valued non-archimedean extension $K$ of $\bb Q_p$ with perfect residue field $k$ and ramification degree $e$. Let $\cal O_C$ be the ring of integers in a complete algebraically closed non-archimedean extension $C$ of $K$ and $X$ be the adic generic fibre of $\bar{\frak X}:=\frak X\times_{{\rm Spf}(\cal O_K)}{\rm Spf}(\cal O_C)$. Assuming $ie<p-1$, 
there is an isomorphism of $\frak S$-modules
\[
H^i_{\frak S}(\frak X_{})\cong H^i_{\rm \acute et}(X, \bb Z_p)\otimes_{\bb Z_p}\frak S.
\]
In particular, we also have an isomorphism of $A_{\inf}$-modules
\[
H^i_{A_{\inf}}(\overline{\frak X})\cong H^i_{\rm \acute et}(X, \bb Z_p)\otimes_{\bb Z_p}A_{\inf}.
\]

\end{theorem}

\begin{proof}
Note that the torsion submodule $H^i_{\frak S-\rm {tor}}$ of $H^i_{\frak S}(\frak X)$ is killed by some power of $p$. Then by Lemma \ref{key} and Lemma \ref{key2}, we get a decomposition $H^i_{\frak S-\rm tor}\cong\bigoplus_{t=1}^n\frak S/p^{m_t}$.

Since $H^i_{\frak S}(\frak X)$ is a direct sum of a free $\frak S$-module and $H^i_{\frak S-\rm {tor}}$ by Corollary \ref{finalcor}, this theorem then follows from the $\rm \acute e$tale specialization of the Breuil--Kisin cohomology groups (see Theorem \ref{theorem1.2})
\[
H^n_{\frak S}(\frak X_{})\otimes_{\frak S}A_{\inf}[1/\mu]\cong H^n_{\rm \acute et}(X, \bb Z_p)\otimes_{\bb Z_p}A_{\inf}[1/\mu].
\]
where the map $\frak S\to A_{\inf}[1/\mu]$ is the composition of the faithfully flat map $\alpha: \frak S\to A_{\inf}$ and the natural injection $A_{\inf}=W(\cal O_C^{\flat})\to A_{\inf}[1/\mu]$.

\end{proof}

\begin{remark}\label{torsionstructure}
In general, for any finitely generated module $M$ over $\frak S$ (or any other two dimensional regular local ring), there is a pseudo-isomorphism between $M$ and $\frak S^r\oplus(\bigoplus_{i=1}^n \frak S/\cal P_i)$ where each $\cal P_i$ is a prime ideal of height $1$. Pseudo-isomorphism means its localization at all prime ideals of height 1 is in fact an isomorphism. Within the range $ie<p-1$, the theorem above tells us that the classical $p$-adic cohomology theories provide enough information to determine the structure of Breuil--Kisin cohomology groups. But beyond this range, the situation gets subtle.
\end{remark}

Now we come to prove the integral comparison theorem in the ramified case.
\begin{theorem}\label{final}
Let $\frak X$ be a proper smooth formal scheme over $\cal O_K$, where $\cal O_K$ is the ring of integers in a complete discretely valued non-archimedean extension $K$ of $\bb Q_p$ with perfect residue field $k$ and ramification degree $e$. Let $\cal O_C$ be the ring of integers in a complete algebraically closed non-archimedean extension $C$ of $K$ with residue field $\bar k$. Let $X$ be the adic generic fibre of $\bar{\frak X}:=\frak X\times_{{\rm Spf}(\cal O_K)}{\rm Spf}(\cal O_C)$ and $\frak X_k$ be the special fiber of $\frak X$. If $ie<p-1$, then there is an isomorphism of $W(k)$-modules
\[
H^i_{\rm \acute et}(X, \bb Z_p)\otimes_{\bb Z_p}W(k)\cong H^i_{\rm crys}(\frak X_{k}/W(k)).
\]

\end{theorem}

\begin{proof}
Assume $ie<p-1$. By Corollary \ref{special} and Corollary \ref{bktor}, we have an isomorphism of $\cal O_C$-modules
\[
H^i_{A_\inf}(\bar{\frak X})/\xi\cong H^i_{\rm dR}(\bar{\frak X}/\cal O_C).
\]

Since we also have $H^i_{A_{\inf}}(\bar{\frak X})\cong H^i_{\rm \acute et}(X, \bb Z_p)\otimes_{\bb Z_p}A_{\inf}$ by Theorem \ref{comp}, we get an isomorphism of $\cal O_C$-modules
\[
H^i_{\rm dR}(\bar{\frak X}/\cal O_C)\cong H^i_{\rm \acute et}(X, \bb Z_p)\otimes_{\bb Z_p}\cal O_C. 
\]
Note that when $e<p$, we have an integral comparison isomorphism between de Rham cohomology and crystalline cohomology (cf. \cite{berthelot2006cohomologie})
\[
H^i_{\rm dR}(\bar{\frak X}/\cal O_C)\cong H^i_{\rm crys}(\frak X_{\bar k}/W(\bar k))\otimes_{W(\bar k)}\cal O_C
\]
where $\frak X_{\bar k}:=\frak X_k\otimes_k\bar k$.

So finally, we get the isomorphism
\[
H^i_{\rm \acute et}(X, \bb Z_p)\otimes_{\bb Z_p}W(\bar k)\cong H^i_{\rm crys}(\frak X_{\bar k}/W(\bar k)).
\]
By virtue of the base change of crystalline cohomology 
\[
H^i_{\rm crys}(\frak X_{\bar k}/W(\bar k))\cong H^i_{\rm crys}(\frak X_{k}/W(k))\otimes_{W(k)}W(\bar k),
\]
we also have
\[
H^i_{\rm \acute et}(X, \bb Z_p)\otimes_{\bb Z_p}W(k)\cong H^i_{\rm crys}(\frak X_{k}/W(k)).
\]

\end{proof}

\begin{remark}
When $(i+1)e<p-1$, the proof of the integral comparison isomorphism for schemes in \cite{caruso2008conjecture} depends on the fact that the crystalline cohomology groups $H^i_{\rm crys}(\frak X_{\cal O_K/p}/S)$ admits a decomposition as $H^i_{\rm crys}(\frak X_{\cal O_K/p}/S)\cong S^{n}\oplus(\bigoplus_{j=1}^mS/p^{a_j})$. This can also be deduced from Theorem \ref{comp} and the base change of prismatic cohomology along the map of prisms $(\frak S, (E))\to (S, (p))$, which is the composition of the Frobenius map $\frak S\to \frak S$ and the natural injection $\frak S\into S$. 
\end{remark}

\section{Categories of Breuil--Kisin modules}\label{extra}\label{F}
In this section, we want to give a slightly more general result about the structure of torsion Breuil--Kisin modules of height $r$, under the restriction $er<p-1$. Namely, all torsion Breuil--Kisin modules in this case are isomorphic to $\bigoplus_{i=1}^n\frak S/p^{a_i}$. As a result, this gives another proof of Theorem \ref{comp} without using Lemma \ref{key}.
~\\

Recall that there is a natural $W(k)$-linear surjection $\beta: \frak S=W(k)[[u]]\to \cal O_K$ sending $u$ to $\pi$. The kernel of this map is generated by an Eisenstein polynomial $E=E(u)$ for $\pi$. Fix a non-negative integer $r$. We first need to define some categories that we will study.
\begin{definition}[$'{\rm{Mod}}^{r,\phi}_{/{\frak S}}$,\cite{caruso2009quasi}]\label{exactcategory}
The objects of category $'{\rm{Mod}}^{r,\phi}_{/{\frak S}}$ are defined to be $\frak S$-modules $\frak M$ equipped with a $\phi$-linear endomorphism $\phi:\frak M\to \frak M$ such that the cokernel of $id\otimes\phi: \phi^{*}\frak M:=\frak S\otimes_{\phi, \frak S}\frak M\to \frak M$ is killed by $E^r$. Morphisms are homomorphisms of $\frak S$-modules compatible with $\phi$. We say that a short sequence $0\to \frak M_1\to \frak M_2\to \frak M_3\to 0$ is exact if it is exact in the abelian category of $\frak S$-modules.
\end{definition}

\begin{definition}[${\rm{Mod}}^{r,\phi}_{/{\frak S_1}}$,\cite{caruso2009quasi}]
The category ${\rm{Mod}}^{r,\phi}_{/{\frak S_1}}$ is the full subcategory of $'{\rm{Mod}}^{r,\phi}_{/{\frak S}}$ spanned by the objects which are finite free over $\frak S_1:=\frak S/p=k[[u]]$.
\end{definition}

\begin{definition}[${\rm{Mod}}^{r,\phi}_{/{\frak S_{\infty}}}$,\cite{caruso2009quasi}]
We define ${\rm{Mod}}^{r,\phi}_{/{\frak S_{\infty}}}$ to be the smallest full subcategory of $'{\rm{Mod}}^{r,\phi}_{/{\frak S}}$ which contains ${\rm{Mod}}^{r,\phi}_{/{\frak S_1}}$ and is stable under extensions. 
\end{definition}
\begin{remark}
The category ${\rm{Mod}}^{1,\phi}_{/{\frak S_1}}$ first appeared in \cite{breuilschemas}. And the category ${\rm{Mod}}^{1,\phi}_{/{\frak S_{\infty}}}$ is just the category $\rm Mod/{\frak S}$ defined by Kisin in \cite{kisin2006crystalline}.
 \end{remark}

The following lemma gives us some important descriptions of objects in ${\rm{Mod}}^{r,\phi}_{/{\frak S_{\infty}}}$.
\begin{lemma}\label{tongliulemma}
\begin{enumerate}
\item
For any $\frak M$ in ${\rm{Mod}}^{r,\phi}_{/{\frak S_{\infty}}}$, the morphism $id\otimes \phi: \phi^*\frak M\to \frak M$ is injective.
\item An object $\frak M$ in $'{\rm{Mod}}^{r,\phi}_{/{\frak S}}$ is in ${\rm{Mod}}^{r,\phi}_{/{\frak S_{\infty}}}$ if and only if it is of finite type over $\frak S$, it has no $u$-torsion and it is killed by some power of $p$.
\end{enumerate}
\end{lemma}
\begin{proof}
See \cite[section 2.3]{liu2007torsion}.
\end{proof}

\begin{corollary}
The torsion submodule $H^i_{\frak S-{\rm tor}}$ of the Breuil--Kisin cohomology groups of a proper smooth formal scheme over $\cal O_K$ is in the category ${\rm{Mod}}^{r,\phi}_{/{\frak S_{\infty}}}$ when $i\leq r<\frac{p-1}{e}$.
\end{corollary}
\begin{proof}
This follow from Corollary \ref{bktor} and \cite[Theorem 1.8 (6)]{bhatt2019prisms}
\end{proof}

Next we introduce Breuil's ring $S$ and define some related categories analogous to those associated with the ring $\frak S$.

\begin{definition}[Breuil's ring]\label{Breuilring}
Let $S$ be the $p$-adic completion of the PD-envelope of $W(k)[u]$ with respect to the ideal $(E)\subset W(k)[u]$. The ring $S$ is endowed with several additional structures:
\begin{enumerate}
\item a canonical (PD-)filtration:  ${\rm Fil^i}S$ is the $p$-adic completion of the ideal generated by elements $(\frac{E^m}{m!})_{m\geq i}$. 
\item a Frobenius $\phi$: it is the unique continuous map which is Frobenius semi-linear over $W(k)$ and sends $u$ to $u^p$.
\end{enumerate}
\end{definition}
For $r<p-1$,  we have $\phi({\rm Fil^r}S)\subset p^rS$ and we can define $\phi_r=\frac{\phi}{p^r}: {\rm Fil^r }S\to S$. Set $S_n:=S/p^n$.

\begin{definition}[$'{\rm Mod}^{r,\phi}_{/S}$,\cite{caruso2009quasi}]
The objects of $'{\rm Mod}^{r,\phi}_{/S}$ are the following data:
\begin{enumerate}
\item an $S$-module;
\item a submodule ${\rm Fil^r}M\subset M$ such that ${\rm Fil^r}S\cdot M\subset {\rm Fil^r}M$;
\item a $\phi$-linear map $\phi_r:{\rm Fil^r}M\to M$ such that for all $s\in {\rm Fil^r}S$ and $x\in M$ we have $\phi_r(sx)=c^{-r}\phi_r(s)\phi_r(E^rx)$, where $c=\phi_1(E)$.
\end{enumerate}
The morphisms are homomorphisms of $S$-modules compatible with additional structures. We say a short sequence $0\to M_1\to M_2\to M_3\to 0$ in $'{\rm Mod}^{r,\phi}_{/S}$ is exact if both sequences $0\to M_1\to M_2\to M_3\to 0$ and $0\to {\rm Fil^r}M_1\to  {\rm Fil^r}M_2\to  {\rm Fil^r}M_3\to 0$ are exact in the abelian category of $S$-modules.
\end{definition}

\begin{definition}[${\rm Mod}^{r,\phi}_{/S_1}$,\cite{caruso2009quasi}]
The objects of ${\rm Mod}^{r,\phi}_{/S_1}$  are $M$ in $'{\rm Mod}^{r,\phi}_{/S}$ such that $M$ is finite free over $S_1$ and the image of $\phi_r$ generates $M$ as an $S$-module.

\end{definition}
\begin{definition}[${\rm Mod}^{r,\phi}_{/S_{\infty}}$,\cite{caruso2009quasi}]\label{classic definition}
The category ${\rm Mod}^{r,\phi}_{/S_{\infty}}$ is the smallest subcategory of $'{\rm Mod}^{r,\phi}_{/S}$ containing ${\rm Mod}^{r,\phi}_{/S_1}$ and is stable under extensions.
\end{definition}

For any $r<p-1$, one can define a functor $M_{\frak S_{\infty}}: {\rm Mod}^{r,\phi}_{/\frak S_{\infty}}\to {\rm 'Mod}^{r,\phi}_{/S}$ as follows: 
\begin{enumerate}
\item $M_{\frak S_{\infty}}(\frak M)=S\otimes_{\phi,\frak S}\frak M$. Here $\phi: \frak S\to S$ is the composite $\frak S\to \frak S\to S$ where the first map is the Frobenius on $\frak S$ and the second map is the canonical injection. 
\item Submodule: The Frobenius on $\frak M$ induces a $S$-linear map $id\otimes \phi: S\otimes_{\phi,\frak S}\frak M\to S\otimes_{\frak S}\frak M$. The submodule ${\rm Fil^r}M_{\frak S_{\infty}}(\frak M)$ is then defined by the following formula:
\[
{\rm Fil^r}M_{\frak S_{\infty}}(\frak M):=\{x\in M_{\frak S_{\infty}}(\frak M) \mid (id\otimes \phi)(x)\in {\rm Fil^r}S\otimes_{\frak S}\frak M\subset S\otimes_{\frak S}\frak M)\}
\]
\item Frobenius: the map $\phi_r$ is the following composite:
\[
{\rm Fil^r}M_{\frak S_{\infty}}(\frak M)\xrightarrow{id\otimes\phi}{\rm Fil^r}S\otimes_{\frak S}\frak M\xrightarrow{\phi_r\otimes id}M_{\frak S_{\infty}}(\frak M).
\]
\end{enumerate}
We state a theorem describing the functor $M_{\frak S_{\infty}}$.

\begin{theorem}\label{classic}
For any $r<p-1$, the functor $M_{\frak S_{\infty}}$ takes value in ${\rm Mod}^{r,\phi}_{/S_{\infty}}$. The induced functor $M_{\frak S_{\infty}}: {\rm Mod}^{r,\phi}_{/\frak S_{\infty}}\to {\rm Mod}^{r,\phi}_{/S_{\infty}}$ is exact and it is an equivalence of categories. Moreover, if we choose $M_{S_{\infty}}$ a quasi-inverse of $M_{\frak S_{\infty}}$, then the functor $M_{S_{\infty}}$ is also exact.
\end{theorem}
\begin{proof}
See \cite[Proposition 2.1.2, Theorem 2.3.1, Proposition 2.3.2]{caruso2009quasi}.
\end{proof}

\begin{theorem}\label{fil}
Assuming $er<p-1$, the category ${\rm Mod}^{r,\phi}_{/S_{\infty}}$ is an abelian category and every object is of the form $\bigoplus_{i=1}^nS/{p^{a_i}}$. For any morphism $f: (M_1, {\rm Fil^r}M_1, \phi_r)\to (M_2, {\rm Fil^r}M_2, \phi_r)$ in ${\rm Mod}^{r,\phi}_{/S_{\infty}}$, the underlying module of ${\rm Ker}(f)$ is the kernel of the morphism $f: M_1\to M_2$ in the category of $S$-modules and the underlying module of ${\rm Fil^r}{\rm Ker}(f)$ is the kernel of the morphism $f: {\rm Fil^r}M_1\to {\rm Fil^r}M_2$ in the category of $S$-modules. A Similar statement is true for ${\rm Coker}(f)$.
\end{theorem}
\begin{proof}
See \cite[Section 3]{caruso2006representations}. We remark that the category which Caruso used is different from ours but they can be proved to be equivalent by using a generalization of \cite[Proposition 2.3.1.2]{breuil1998cohomologie}, as mentioned in the proof of \cite[Theorem 4.2.1]{caruso2008conjecture}.
\end{proof}
\begin{remark}
This theorem is false without the restriction $er<p-1$.
\end{remark}

From now on, we fix a non-negative integer $r$ such that $er<p-1$. Then ${\rm Mod}^{r,\phi}_{/\frak S_{\infty}}$ is an abelian category. 

\begin{lemma}
For any morphism $f: \frak M_1\to \frak M_2$ in ${\rm Mod}^{r,\phi}_{/\frak S_{\infty}}$, the underlying module of ${\rm Ker}(f)$ is the kernel of the morphism $f: \frak M_1\to \frak M_2$ in the category of $\frak S$-modules. A Similar statement is true for ${\rm Coker}(f)$.
\end{lemma}

\begin{proof}
By Lemma \ref{tongliulemma}, the kernel and the image of the underlying morphism $f: \frak M_1\to \frak M_2$ in the category of $\frak S$-modules together with the induced Frobenius maps are objects of ${\rm Mod}^{r,\phi}_{/\frak S_{\infty}}$. It is easy to see that the kernel equipped with the induced Frobenius map is indeed ${\rm Ker}(f)$ in the category ${\rm Mod}^{r,\phi}_{/\frak S_{\infty}}$. So we can assume $f: \frak M_1\to \frak M_2$ is injective. Then $M_{\frak S_{\infty}}(f)$ is also injective. In fact, let $L$ be the kernel of $M_{\frak S_{\infty}}(f)$ and we choose a quasi-inverse functor $M_{S_{\infty}}$ of $M_{\frak S_{\infty}}$. Let $h: \frak L\to \frak M_1$ be the image of the inclusion $L\to M_{\frak S_{\infty}}(\frak M_1)$ under $M_{S_{\infty}}$. Then $f\circ h=0$, which implies $h=0$. In consequence, we have $L=0$. Put $M={\rm Coker}(f)$. By Theorem \ref{classic} and Theorem \ref{fil}, we get an exact sequence $0\to\frak M_1\to \frak M_2\to M_{S_{\infty}}(M)\to 0$ in the exact category ${\rm Mod}^{r,\phi}_{/\frak S_{\infty}}$ (where the class of the exact sequences is as defined in Definition \ref{exactcategory}). So we have $M_{S_{\infty}}(M)$ is isomorphic to $\frak M_2/\frak M_1$ as $\frak S$-modules. In particular $\frak M_2/\frak M_1$ has no $u$-torsion. By Lemma \ref{tongliulemma}, the module $\frak M_2/\frak M_1$ equipped with the induced Frobenius map is an object of ${\rm Mod}^{r,\phi}_{/\frak S_{\infty}}$. It is easy to check that ${\rm Coker}(f)$ is isomorphic to $\frak M_2/\frak M_1$ equipped with the induced Frobenius map.
\end{proof}

\begin{corollary}\label{rrrrr}
The full subcategory ${\rm Mod}^{r,\phi}_{/\frak S_1}$ of ${\rm Mod}^{r,\phi}_{/\frak S_{\infty}}$ is an abelian category.
\end{corollary}
\begin{proof}
For any morphism $f: \frak M_1\to \frak M_2$ in ${\rm Mod}^{r,\phi}_{/\frak S_1}$, ${\rm Ker}(f)$ and ${\rm Coker}(f)$ are then both killed by $p$. By Lemma \ref{tongliulemma}, they are $u$-torsion free. So ${\rm Ker}(f)$ and ${\rm Coker}(f)$ are in the category ${\rm Mod}^{r,\phi}_{/\frak S_1}$.
\end{proof}

Let ${\rm ModFI}^{r,\phi}_{/\frak S_{\infty}}$ denote the full subcategory of ${\rm Mod}^{r,\phi}_{/\frak S_{\infty}}$ spanned by the objects that are isomorphic to $\bigoplus_{i=1}^{n} \frak S/p^{a_i}$ as $\frak S$-modules. In particular, ${\rm ModFI}^{r,\phi}_{/\frak S_{\infty}}$ contains ${\rm Mod}^{r,\phi}_{/\frak S_1}$.

\begin{lemma}\label{modp}
For any $\frak M\in {\rm Mod}^{r,\phi}_{/\frak S_{\infty}}$, the quotient $\frak M/p$ is in ${\rm Mod}^{r,\phi}_{/\frak S_1}$.
\end{lemma}
\begin{proof}
Consider the morphism $\frak M\xrightarrow{\times p}\frak M$ in ${\rm Mod}^{r,\phi}_{/\frak S_{\infty}}$. Since ${\rm Mod}^{r,\phi}_{/\frak S_{\infty}}$ is an abelian category, we know that $\frak M/p$ is also in ${\rm Mod}^{r,\phi}_{/\frak S_{\infty}}$. It is killed by $p$ and has no $u$-torsion by Lemma \ref{tongliulemma}, therefore $\frak M/p$ is in ${\rm Mod}^{r,\phi}_{/\frak S_1}$.
\end{proof}

We now reformulate Lemma \ref{2mod} by using the categories we have defined.
\begin{lemma}\label{mod}
Let $\frak M$ be in ${\rm Mod}^{r,\phi}_{/\frak S_{\infty}}$. If $p\frak M$ is in ${\rm ModFI}^{r,\phi}_{/\frak S_{\infty}}$, so is $\frak M$.
\end{lemma}
\begin{proof}
By Lemma \ref{modp}, we have $\frak M/p\in {\rm Mod}^{r,\phi}_{/\frak S_{\infty}}$. Then this lemma follows from Lemma \ref{2mod}.
\end{proof}
\begin{lemma}\label{inj}
Let $\frak L\hookrightarrow \frak M$ be an injection in ${\rm Mod}^{r,\phi}_{/\frak S_{\infty}}$. If $\frak M$ is in ${\rm ModFI}^{r,\phi}_{/\frak S_{\infty}}$, so is $\frak L$.
\end{lemma}
\begin{proof}
We show that $p\frak L$ is in ${\rm ModFI}^{r,\phi}_{/\frak S_{\infty}}$, then this lemma follows from Lemma \ref{mod}. Consider the map $p\frak L\hookrightarrow p\frak M$. We proceed by induction on the minimal integer such that $p^n\frak M=0$. If $n=1$, this is easy. Assume that when $n<m$ this lemma is true. Then when $n=m$, $p\frak L$ is also in ${\rm ModFI}^{r,\phi}_{/\frak S_{\infty}}$ as $p^{m-1}(p\frak M)=0$. We are done.
\end{proof}

\begin{theorem}
The category ${\rm ModFI}^{r,\phi}_{/\frak S_{\infty}}$ is an abelian category.
\end{theorem}
\begin{proof}
For any morphism $f:\frak M_1\to \frak M_2$ in ${\rm ModFI}^{r,\phi}_{/\frak S_{\infty}}$, we need to show $\frak L={\rm Ker}(f)$ and $\frak C={\rm Coker}(f)$ are also in the category ${\rm ModFI}^{r,\phi}_{/\frak S_{\infty}}$. For the kernel $\frak L$, this follows from Lemma \ref{inj}. For the cokernel $\frak C$, we proceed by induction on the minimal integer $n$ such that $p^n\frak M_2=0$. Without loss of generality, we can assume $f$ is an injection. 

When $n=1$, we have $\frak M_1, \frak M_2$ are both in ${\rm Mod}^{r,\phi}_{/\frak S_1}$. Then by Corollary \ref{rrrrr}, we see that $\frak C$ is also in ${\rm Mod}^{r,\phi}_{/\frak S_1}\subset {\rm ModFI}^{r,\phi}_{/\frak S_{\infty}}$. Now suppose the statement is true when $n<m$. When $n=m$, consider the sequence $p\frak M_1\to p\frak M_2\to p\frak C$. Then there is a short exact sequence $0\to \frak L^{'}\to p\frak M_2/{p\frak M_1}\to p\frak C\to 0$. Since $p^{m-1} (p\frak M_2/{p\frak M_1})=0$, by the assumption, we get $p\frak C$ is in ${\rm ModFI}^{r,\phi}_{/\frak S_{\infty}}$. Then by Lemma \ref{mod}, we see that $\frak C$ is also in ${\rm ModFI}^{r,\phi}_{/\frak S_{\infty}}$. This finishes the proof.
\end{proof}

\begin{theorem}\label{newproof}
There is an equivalence of categories: ${\rm ModFI}^{r,\phi}_{/\frak S_{\infty}}$=${\rm Mod}^{r,\phi}_{/\frak S_{\infty}}$. 
\end{theorem}
\begin{proof}
We just need to prove that every object $\frak M$ in ${\rm Mod}^{r,\phi}_{/\frak S_{\infty}}$ is also in ${\rm ModFI}^{r,\phi}_{/\frak S_{\infty}}$. To see this, we proceed by induction on the minimal integer $n$ such that $p^n\frak M=0$.

When $n=1$, this follows from Lemma \ref{mod}. Now suppose the statement is true when $n<m$. Then when $n=m$, we know that $p\frak M$ is killed by $p^{m-1}$. So by the assumption, we have $p\frak M\in {\rm ModFI}^{r,\phi}_{/\frak S_{\infty}}$. By Lemma \ref{mod}, we can obtain that $\frak M\in {\rm ModFI}^{r,\phi}_{/\frak S_{\infty}}$. We are done.
\end{proof}

So Theorem \ref{newproof} provides another proof of Theorem \ref{comp}.
\begin{theorem}\label{decom}
For any $i\leq r<\frac{p-1}{e}$, we have $H^i_{\frak S-{\rm tor}}$, the torsion submodule of the Breuil--Kisin cohomology group of a proper smooth formal scheme over $\cal O_K$, is in the category ${\rm{Mod FI}}^{r,\phi}_{/{\frak S_{\infty}}}$, i.e. $H^i_{\frak S-{\rm tor}}\cong \bigoplus_{i=1}^n\frak S/p^{a_i}$.
\end{theorem}

\bibliographystyle{alpha} 

\bibliography{padichodge.bib}

\end{document}